\documentclass[reqno,usletter]{amsart}
\usepackage[T1]{fontenc}
\usepackage[latin1]{inputenc}
\usepackage{amssymb}
\usepackage{amsfonts}
\usepackage{amsmath,mathtools}
\usepackage{mathrsfs}
\usepackage{graphicx}
\usepackage{color}
\usepackage{enumitem}
\usepackage{hyperref}
\usepackage{esint}
\usepackage{verbatim}
\usepackage{bbm}
\usepackage{tikz}
\newcommand{\M}[1]{\mathbb{M}^{#1}}
\newcommand{\N}[1]{\mathbb{N}^{#1}}
\newcommand{\R}[1]{\mathbb{R}^{#1}}
\renewcommand{\S}[1]{\mathbb{S}^{#1}}

\newcommand{\Ad}{\mathsf{Ad}}

\newcommand{\cB}{\mathcal B}
\newcommand{\cC}{\mathcal C}

\newcommand{\cH}{\mathcal H}

\newcommand{\cL}{\mathcal L}
\newcommand{\cM}{\mathcal M}

\newcommand{\sI}{\mathscr{I}}

\newcommand{\fC}{\mathfrak C}

\newcommand{\bulk}{\mathrm{bulk}}

\newcommand{\de}{\mathrm d}

\newcommand{\surface}{\mathrm{surface}}

\newcommand{\eps}{\varepsilon}

\newcommand{\dist}[2]{\operatorname{dist}(#1,#2)}

\newcommand{\spt}{\operatorname{spt}}
\newcommand{\tr}{\operatorname{tr}}

\newcommand{\blu}[1]{\textcolor[rgb]{0,0,1}{#1}}

\renewcommand{\geq}{\geqslant}
\renewcommand{\leq}{\leqslant}

\newcommand{\wsto}{\stackrel{*}{\rightharpoonup}}

\newcommand{\average}{{\mathchoice {\kern1ex\vcenter{\hrule
height.4pt width 8pt depth0pt}
\kern-11pt} {\kern1ex\vcenter{\hrule height.4pt width 4.3pt
depth0pt} \kern-7pt} {} {} }}

\newcommand{\res}{\mathop{\hbox{\vrule height 7pt width .5pt depth
0pt\vrule height .5pt width 6pt depth0pt}}\nolimits}

\mathchardef\emptyset="001F

\providecommand{\U}[1]{\protect\rule{.1in}{.1in}}
\numberwithin{equation}{section}
\def\e{\eps}

\newtheorem{definition}{Definition}[section]
\newtheorem{theorem}[definition]{Theorem}
\newtheorem{lemma}[definition]{Lemma}
\newtheorem{proposition}[definition]{Proposition}

\newtheorem{corollary}[definition]{Corollary}

\theoremstyle{definition} {\newtheorem{remark}[definition]{Remark}}

\makeindex

\begin{document}

\author{Jos\'{e} Matias}
\address{Departamento de Matem\'atica, Instituto Superior T\'ecnico, Av.~Rovisco Pais, 1, 1049-001 Lisboa, Portugal}
\email[J.~Matias]{jose.c.matias@tecnico.ulisboa.pt}
\author{Marco Morandotti}
\address{Dipartimento di Scienze Matematiche ``G.~L.~Lagrange'', Politecnico di Torino, Corso Duca degli Abruzzi, 24, 10129 Torino, Italy}
\email[M.~Morandotti]{marco.morandotti@polito.it}
\author{David R.\@ Owen}
\address{Department of Mathematical Sciences, Carnegie Mellon University, 5000 Forbes Ave., Pittsburgh, 15213 USA}
\email[D.~R.~Owen]{do04@andrew.cmu.edu}
\author{Elvira Zappale}
\address{Dipartimento di Ingegneria Industriale, Universit\`{a} degli Studi di Salerno, Via Giovanni Paolo II, 132, 84084 Fisciano (SA), Italy}
\email[E.~Zappale]{ezappale@unisa.it}

\subjclass[2010]{49J45 
(74G65, 
74A60, 
74C99, 
74N05)} 

\title[Upscaling of non-local energies]{Upscaling and spatial localization of non-local energies with applications to crystal plasticity} 
\date{\today}

\begin{abstract}
We describe multiscale geometrical changes via structured deformations $(g,G)$ and the non-local energetic response at a point $x$ via a function $\Psi$ of the weighted averages of the jumps $[u_{n}](y)$ of microlevel deformations $u_{n}$ at points $y$ within a distance $r$ of $x$. 
The deformations $u_{n}$ are chosen so that $\lim_{n\to \infty }u_{n}=g$ and $\lim_{n\to \infty }\nabla u_{n}=$ $G$. 
We provide conditions on $\Psi$ under which the upscaling ``$n\to \infty$'' results in a macroscale energy that depends through $\Psi$ on (1) the jumps $[g]$ of $g$ and the ``disarrangment field'' $\nabla g-G$, (2) the ``horizon'' $r$, and (3) the weighting function $\alpha _{r}$ for microlevel averaging of  $[u_{n}](y)$. 
We also study the upscaling ``$n\to \infty$'' followed by spatial localization ``$r\to 0$'' and show that this succession of processes results in a purely local macroscale energy $I(g,G)$ that depends through $\Psi$ upon the jumps $[g]$ of $g$ and the ``disarrangment field'' $\nabla g-G$, alone. 
In special settings, such macroscale energies $I(g,G)$ have been shown to support the phenomena of yielding and hysteresis, and our results provide a broader setting for studying such yielding and hysteresis.
As an illustration, we apply our results in the context of the plasticity of single crystals.

\end{abstract}

\maketitle

\noindent \textbf{Keywords:} structured deformations, upscaling, relaxation, spatial localization, non-local energies, crystal plasticity.

\tableofcontents

\section{Introduction}\label{section:intro}

The phenomena of yielding, hysteresis, and dissipation are central aspects of the plastic behavior of solids. 
The research described in this article addresses the energetics underlying such phenomena and builds on the approach taken in two earlier publications \cite{CDPFO1999,DO2002} in this journal.
In \cite{CDPFO1999} an enriched energetics \cite{CF1997}, attained by means of the multiscale geometry of structured deformations \cite{DPO1993}, led to the identification of instabilities due to submacroscopic slips and to the prediction of yielding, hysteresis, and dissipation within a class of structured deformations called two-level shears. 
In \cite{DO2002}, again within the class of two-level shears, the additional phenomenon of work-hardening was shown to emerge from a refinement of the approach in \cite{CDPFO1999}, and the quantitative details of the hardening property of the resulting constitutive response were identified by means of the classical experiments of G.~I.~Taylor on aluminum single crystals in soft devices.

Key elements in the prediction of yielding, hysteresis, and dissipation that emerged from the approach taken in \cite{CDPFO1999} and \cite{DO2002} are
\begin{itemize}
\item[(i)] an additive decomposition of the amount $\mu$ of macroscopic shear into the \emph{shear without slip} $\gamma $, the amount of shear at the macrolevel due to submacroscopic shears away from slip bands diffused throughout the crystal, and into the \emph{shear due to (micro)slip} $\mu-\gamma$, the amount shear at the macrolevel due to the relative displacements of portions of the crystal on opposite sides of slip bands;
\item[(ii)] a corresponding additive decomposition of the volume density of free energy $H(\mu ,\gamma )$ into a part $\varphi (\gamma )$, a strictly convex measure of the volume density of energy stored by the crystal lattice away from slip bands, and a part $\psi (\mu -\gamma )$, a periodic measure of the volume density of energy stored due to the relative displacement of the portions of the crystal on opposite sides of slip bands:
\begin{equation}\label{decomposition for energy}
H(\mu ,\gamma )=\varphi (\gamma )+\psi (\mu -\gamma ).
\end{equation}
The periodicity of the part $\psi $ due to crystallographic slips is the source of an infinite family of metastable branches of the set of pairs $(\mu ,\gamma )$ that render stationary $H(\mu ,\gamma )$, in the case of a hard device, or $H(\mu ,\gamma )-\sigma \mu $ in the case of a soft device (with $\sigma $ a given applied shear stress).
The analysis in \cite{CDPFO1999} shows that on each stable branch, reversible loading and unloading is available, while irreversible jumps between the ends of metastable branches are the source of yielding and hysteresis.
\end{itemize}

The decomposition (i) in the context of two-level shears is an instance of the more general additive decomposition \cite{DPO1993} associated with a structured deformation $(g,G)$ in a fully three-dimensional context:
\begin{equation}\label{G + M}
\nabla g=G+M,  
\end{equation}
a decompostion of the macroscopic deformation gradient $\nabla g$ into the part $G$, the deformation without submacroscopic disarrangements, and the part $M$, the deformation due to disarrangements. 
This terminology is justified by demonstrating the existence of piecewise smooth approximations $u_{n}$ of the macroscopic deformation $g$ with the properties
\begin{equation}\label{identification relations}
g=\lim_{n\to \infty }u_{n},\qquad G=\lim_{n\to\infty}\nabla u_n
\end{equation}%
and with the property that $M=\nabla g-G$ is a limit of averages of the jumps $[u_{n}]$ of the approximations. 
In the context of \cite{DPO1993}, the limits in \eqref{identification relations} are taken in the sense of essentially uniform convergence, and the accommodation inequality
\begin{equation}\label{accommodation inequality}
0<\det G\leq \det \nabla g, 
\end{equation}%
a defining requirement in \cite{DPO1993} for a structured deformation, assures that \emph{injective} approximations $u_{n}$ of the injective mapping $g$ are available in \eqref{identification relations}. 
The scalar fields $\gamma $ and $\mu -\gamma $ for two-level shears in (i) are components of the tensor fields $G$ and $M$, respectively, in \cite{DPO1993}.

The decomposition (ii) in the context of two-level shears was motivated by the seminal study of the energetics of structured deformations provided in \cite{CF1997}. 
The structured deformations $(g,G)$, studied there in a multidimensional setting appropriate for variational analysis and described here in Section~\ref{sec:prel_SD}, were shown to lead to decompositions and identification relations analogous to \eqref{identification relations}, but the unavailability of \eqref{accommodation inequality} in that setting means that there is lacking a guarantee that injective approximations $u_{n}$ as in \eqref{identification relations} exist. 
Nevertheless, the broader setting permits one to assign an energy $E_{L}(u_{n})$ to the approximating functions $u_{n}$ in \eqref{identification relations} that is the sum of a bulk part depending upon $\nabla u_{n}$ and an interfacial part depending upon $[u_{n}]$. 
In addition, in this setting one also can identify a ``relaxed'', or ``upscaled'', energy $I(g,G)$ by means of optimally chosen approximations $u_{n}$. 
More significantly, the analysis in \cite{CF1997} shows that the upscaled energy is the sum of a bulk part that depends upon $\nabla g$ and $G$ (or, equivalently, upon $M$ and $G$) and of an interfacial part that depends upon $[g]$. Consequently, the identification relations \eqref{identification relations} tell us that the bulk energy density $H(\nabla g,G)$ that determines the bulk part of the upscaled energy $I(g,G)$ reflects both the smooth part $\nabla u_{n}$ and the disarrangements $[u_{n}] $ of the approximations.

A volume density of free energy for two-level shears that depends only upon $\mu$ and $\gamma$ would be a one-dimensional counterpart of the bulk density $H(\nabla g,G)$ in \cite{CF1997}. 
However, in the setting of \cite{CF1997}, an additive decomposition of $H(\nabla g,G)$ of the type \eqref{decomposition for energy} into a part depending only upon the deformation without disarrangements $G$ and a part depending only upon the deformation due to disarrangements $M$ generally is not available (see the discussion in Subsection~\ref{sect7}). 
Moreover, in the setting of \cite{CF1997}, in special cases in which such a decomposition is available, the further availability of a dependence on $M$ for the upscaled energy that reflects the periodicities of a crystal lattice is not assured. 
Consequently, the upscaling of energy in \cite{CF1997} provides ingredients necessary, but not sufficient, for the justification of the decomposition \eqref{decomposition for energy} with the stated properties of the terms $\varphi $ and $\psi $ used in \cite{CDPFO1999} as well
as in \cite{DO2002}.

The determination of an upscaling of energy that justifies the particular decomposition \eqref{decomposition for energy} was studied in \cite{DPO2000} in a one-dimensional setting by adding to the energy $E_{L}(u_{n})$ a term with a non-local dependence on the jumps $[u_{n}]$ of the approximations $u_{n}$ in \eqref{identification relations}. 
The additional non-local term 
\begin{equation}\label{non-local term}
\int_{0}^{1}\Psi \bigg(\sum_{z\in S_{u_{n}}\cap (x-r,x+r)}\frac{[u_{n}](z)}{2r}\bigg)\,\de x
\end{equation}
introduced in \cite{DPO2000} involves the values of an energy density $\Psi $ at the average  $\sum_{z\in S_{u_{n}}\cap (x-r,x+r)}\frac{[u_{n}](z)}{2r}$ of the jumps $[u_{n}](z)$ of the approximation $u_{n}$ at all points $z\in S_{u_{n}}$, the set of jump points of $u_{n}$, that lie within the interval $(x-r,x+r)$; the energy density then is integrated over the interval $(0,1)$ representing the body under consideration. 
The main result obtained there for a particular class of energy densities $\Psi $ is the formula
\begin{equation}\label{1d upscaling localization}
\lim_{r\rightarrow 0}\lim_{n\to \infty }\int_{0}^{1}\Psi\bigg(\sum_{z\in S_{u_{n}}\cap (x-r,x+r)}\frac{[u_{n}](z)}{2r}\bigg)\,\de x=\int_{0}^{1}
\Psi (M(x))\,\de x  
\end{equation}
that holds for sequences of approximations $u_{n}$ satisfying \eqref{identification relations} and whose jumps are all of the same sign. 
The iterated limit on the left-hand side of this formula is interpreted as the operation of upscaling, ``$\lim_{n\to \infty}$'', followed by the operation of spatial localization, ``$\lim_{r\to0}$'', and is seen to result in a bulk energy $\int_{0}^{1}\Psi (M(x))\,\de x$ in which the original energy density $\Psi $ is evaluated at the disarrangement tensor $M(x)=\nabla g(x)-G(x)$ for the given structured deformation $(g,G)$. 
In particular, properties of the energy density $\Psi $ such as periodicity persist under the operations of upscaling and localization. 
Moreover, sufficient conditions were provided in \cite{DPO2000} in order that the total energy obtained by addition of the non-local term \eqref{non-local term} to the energy $E_{L}(u_{n})$ satisfies
\begin{equation}\label{upscaling of sum}
\lim_{r\to 0}\lim_{n\to \infty}\bigg\{E_{L}(u_{n})+\int_{0}^{1}\Psi \bigg(\sum_{z\in S_{u_{n}}\cap (x-r,x+r)}\frac{[u_{n}](z)}{2r}\bigg)\,\de x\bigg\} =\lim_{n\to \infty }E_{L}(u_{n})+\int_{0}^{1}\Psi (M(x))\,\de x  
\end{equation}
and eventually assures that a decomposition of the form \eqref{decomposition for energy} is available in the case of two-level shears. 
Thus, the use of the operations of upscaling followed by localization in a one-dimensional setting justifies the decomposition \eqref{decomposition for energy} with the stated properties of the terms $\varphi $ and $\psi $ used in \cite{CDPFO1999} as well as in \cite{DO2002} and thereby provides a firmer basis for the analyses in \cite{CDPFO1999} and \cite{DO2002} that predict yielding, hysteresis, dissipation, and hardening in the context of two-level shears.

Our main goal in this paper is to justify a fully three-dimensional analogue of the decomposition \eqref{decomposition for energy} that will provide in future research a basis for an analysis of yielding, hysteresis, dissipation, and hardening in a context far broader than that of two-level shears. 
Our efforts here to this end consist of 
\begin{enumerate}[label=(\alph*)]
\item\label{(a)} the broadening of the analysis of upscaling and localization in \cite{DPO2000} to a class of structured deformations that includes the genuinely three-dimensional changes in geometry encountered in the deformations of single crystals, and 
\item\label{(b)} the provision of a context compatible with the analysis in \ref{(a)} in which the energy density $\Psi $ has properties that reflect underlying periodicities of a crystalline lattice and, in addition, has the property of frame-indifference, central in the formulation of constitutive assumptions in continuum mechanics. 
\end{enumerate}
The results in this paper thus provide the desired energetic and geometrical starting point for a fully three-dimensional study of yielding, hysteresis, and dissipation in single crystals in the context of structured deformations.

Item~\ref{(a)} is addressed in Sections~\ref{section:mathform} through~\ref{together} below. 
Following a summary in Section~\ref{section:preliminaries} of the required mathematical background on measure theory and the $SBV$ theory of structured deformations, we state in Section~\ref{section:mathform} our main results on the upscaling and spatial localization of a non-local, multidimensional analog of the one-dimensional non-local term energy \eqref{non-local term}, as well as, in Section~\ref{together}, our multidimensional analog of the one-dimensional formula \eqref{upscaling of sum} that pertains to a sum of local and non-local energies. 
The non-local analog of \eqref{non-local term} that we consider is a weighted average of jumps in which the weighting function is smooth and vanishes outside of a ball of radius $r$ centered at a given point $x$ in the body. 
This type of average is amenable to more standard mathematical tools, and we leave for future study whether or not the case of a discontinuous weighting function, one such as in \eqref{non-local term}, which is a non-zero constant inside the ball and vanishes outside the ball, is amenable to mathematical analysis in the present context.
The proofs and verifications of our theorems and formulas are provided in Section~\ref{section:new3} for the analysis of the non-local term and in Section~\ref{together} for the analysis of the sum of local and non-local terms. 
Our multidimensional analogs of both local and non-local energies allow for an explicit dependence of the energy densities on position. 

The explicit dependence on position just mentioned is necessitated by the additional requirement of frame-indifference of energetic response that is imposed in Section~\ref{plasticity} when we apply our results to crystal plasticity and by the transformation properties of the deformation due to disarrangements $M$ under changes of observer. 
There we provide the context mentioned in item~\ref{(b)} by restricting our attention to a more limited class of structured deformations suitable for describing deformations of single crystals, namely, to invertible structured deformations \cite{DPO1993,DO2002a}. 
For invertible structured deformations, the fields $g$ and $G$ have additional smoothness, and the accommodation inequality \eqref{accommodation inequality} is satisfied with equality, so that the macroscopic volume change $\det \nabla g$ equals the the volume change $\det G$ due to smooth submacroscopic changes. 
Consequently, the submacroscopic disarrangements associated with invertible structured deformations cannot entail any volume change, as would occur during the formation of microvoids, and so are compatible with the submacoscopic slips observed in single crystals. 
In Section~\ref{plasticity} we describe further a class of ``crystallographic structured deformations'', invertible structured deformations whose disarrangement tensor $M$ reflects the availability of submacroscopic slips only on specific crystallographic planes in specific crystallographic directions. 
The bulk of the analysis in Section~\ref{plasticity} then is devoted to showing that it is appropriate in the context of crystallographic structured deformations to require that the energy density $\Psi $ has properties that reflect underlying periodicities of a crystalline lattice, as well as the property of frame-indifference. 
As noted in the last paragraph of Section~\ref{plasticity}, when an energy density with these properties is used to generate the non-local term studied in Sections~\ref{section:mathform} and~\ref{section:new3}, the results from Section~\ref{together} for the analysis of the sum of local and non-local terms provide the desired multidimensional version of \eqref{decomposition for energy} appropriate for crystallographic deformations of single crystals, the starting point for a future study of yielding, hysteresis, dissipation and work-hardening.
Readers mainly interested in applications to plasticity may wish to focus on Sections~\ref{section:mathform} and~\ref{plasticity}, referring to the mathematical preliminaries in Section~\ref{section:preliminaries} as needed.

Our approach to the study of non-local energies rests on two limiting
processes:
\begin{description}
\item[upscaling] Starting from a submacroscopic level at which a weighted average of disarrangements within each neighborhood of a fixed size $r>0$ determines the non-local energy, one passes to the macrolevel, permitting disarrangements to diffuse throughout each such neighborhood.
This upscaling process determines for each given structured deformation the dependence of the upscaled energy density on that structured deformation.
\item[spatial localization] Starting at the macrolevel from neighborhoods of the given size $r$ above, one passes to neighborhoods of smaller and smaller sizes to obtain in the limit $r\to 0$ purely local bulk and interfacial energy densities for the structured deformation in question. 
\end{description}

The results in Sections~\ref{section:mathform} through~\ref{together} show that the non-linearities in the non-local energy embodied in the choice of energy density $\Psi $, before upscaling, persist under these two limiting processes. 
In cases that can be identified through our analysis in those sections, the mathematical effect of the upscaling and spatial localization is to replace the initial weighted average of jumps $[u_{n}]$ of approximating deformations at which $\Psi $ is to be evaluated by the disarrangement tensor $M=\nabla g-G$ for the given structured deformation that $u_{n}$ approximates.

To our knowledge, previous research that involves upscaling or spatial localization of energies has been restricted to the consideration of only
one of these two limiting processes, rather than both.
 In \cite{nogap,CF1997,DP2001,DPO2000,OP2000,S17} the upscaling is carried out for purely local energy densities, so that the parameter $r$ does not appear and the spatial localization is irrelevant. 
 The important results in \cite{S17} that are exemplified in \cite{nogap,CF1997,DP2001,DPO2000,OP2000} show that the bulk energies obtained via upscaling, when the energy before upscaling is both local and purely interfacial, form a class that excludes the periodic functions used in \cite{CDPFO1999,DO2002} to predict yielding and hysteresis. 
 Peridynamics provides a context in which only spatial localization is employed: for example, instances of classical, local theories of elasticity and fracture are recoverd in \cite{SL2008} and \cite{L2016} from peridynamics under the limiting process $r\to 0$. 
In the case of peridynamics, the principal focus with respect to storage of energy and with respect to associated field theories is the non-local case in which the ``horizon'' $r$ remains fixed.

To address here the extensive published research on crystal plasticity that employs neither the explicit use of upscaling nor that of spatial localization, we mention first the article \cite{LandauTheory}.
The introduction in that recent paper provides an extensive bibliography of such research, and the substantial analysis in two dimensions provided there is based in part on the availability of an infinite family of energy wells (as was the case in \cite{CDPFO1999} and \cite{DO2002}) and on the requirement of frame-indifference, ingredients that will enter into a future analysis of yielding, hysteresis, and dissipation, based on the results on energetics of structured deformations in the present paper.
Another recent article \cite{ESJB2013} provides a critical review of an essential element of ``physically based'' plasticity theories and gives an alternative approach to conventional treatments.

\section{Mathematical preliminaries}\label{section:preliminaries}
We start this section by fixing the notation used throughout this work; then we recall the basic notions and definitions of $SBV$ functions, which are essential to introduce the variational setting of \cite{CF1997} for structured deformations, and we conclude by recalling some results on measure theory, which will be useful to state our main results.
In this section, we will use the symbol $U$ to denote an open subset of $\R{N}$ with boundary $\partial U$ of zero Lebesgue measure; we save the symbol $\Omega$ for a bounded, connected, open subset of $\R{N}$ with Lipschitz boundary which represents the body undergoing the mechanical deformations we are interested in describing and characterizing. Since both general results and ones specific to our mechanical setting are presented, we will distinguish between functions or measures defined on the set $U$ taking values in $\R{\ell}$, for a generic $\ell\in\N{}$ (for the general case), and functions or measures defined on the set $\Omega$ taking values in $\R{d\times N}$, for $N,d\in\N{}$ (for those carrying a mechanical meaning). The properties that we require of the set $\Omega$ are stronger than those required on $U$, so that all of the properties valid for functions $U\to\R{\ell}$ will be valid, in particular, for functions $\Omega\to\R{d\times N}$.
\subsection{Notation}\label{notation}
We will use the following notations
\begin{itemize}
\item[-] $\cL^{N}$ and $\cH^{N-1}$ denote the  $N$-dimensional Lebesgue measure and the $\left(  N-1\right)$-dimensional Hausdorff measure in $\R N$, respectively; the symbol $\de x$ will also be used to denote integration with respect to $\cL^{N}$, while $\de \cH^{N-1}$ will be used to denote surface integration with respect to $\cH^{N-1}$;
\item[-] $U \subset \R{N}$ is an open set with $\cL^N(\partial U) = 0$;
\item[-] $\Omega \subset \R{N}$ is a bounded connected open Lipschitz set; 
\item[-] ${\mathcal A}(\Omega)$ is the family of all open subsets of $\Omega $; $\cB(\Omega)$ is the family of all Borel subsets of $\Omega$;
\item [-] $C^\infty_c(U;\R\ell)\coloneqq\{u\colon U\to\R{\ell}: \text{$u$ is smooth and has compact support in $U$}\}$; if $\ell=1$, we just denote this set by $C^\infty_c(U)$;
$C^\infty_0(U)$ and $C^\infty_0(U;\R\ell)$ denote the closures of $C^\infty_c(U)$ and $C^\infty_c(U;\R\ell)$, respectively, in the $\sup$ norm; 
the same definitions are given if we replace $U$ by $\overline U$;
\item [-] $C_c(U;\R\ell)\coloneqq\{u\colon U\to\R{\ell}: \text{$u$ is continuous and has compact support in $U$}\}$; if $\ell=1$, we just denote this set by $C_c(U)$; 
$C_0(U)$ and $C_0(U;\R\ell)$ denote the closures of $C_c(U)$ and $C_c(U;\R\ell)$, respectively, in the $\sup$ norm; 
the same definitions are given if we replace $U$ by $\overline U$; observe that $C_0(\overline\Omega;\R\ell)=C(\overline\Omega;\R\ell)$;
\item[-] $C(U;\R\ell)\coloneqq\{u\colon U\to\R{\ell}: \text{$u$ is continuous on $U$}\}$; 
the same definitions are given if we replace $U$ by $\overline U$;
\item[-] $\cM (U)$ and $\cM(U;\mathbb R^{\ell})$ are the sets of (signed) finite real-valued or vector-valued Radon measures on $U$, respectively; $\cM ^+(U)$ is the set of non-negative finite Radon measures on $U$;
\item[-] $\cM(\overline U)$ and $\cM(\overline U;\mathbb R^{\ell})$ denote the dual spaces of the sets $C_{\blu{0}}(\overline U)$ and $C_{\blu{0}}(\overline U;\R{\ell})$ of continuous functions, respectively; 
\item[-] given $\mu\in\cM(U)$ (or $\cM(\overline U))$ or $\mu\in\cM(U;\R{\ell})$ (or $\cM(\overline U; \R{\ell})$), the measure $|\mu|\in\cM^+(U)$ (or $\cM^+(\overline U)$) denotes the total variation of $\mu$;
\item[-] given $\mu\in\cM(U;\R{d})$ (or $\cM(\overline U;\R{d})$), we denote by $\mu=m^a\cL^N+\mu^s$ its decomposition into absolutely continuous part with respect to the Lebesgue measure and singular part; for every $A\in\cB(U)$,(or $\cB(\overline U))$ we define $\langle\mu\rangle(A)\coloneqq \int_A \sqrt{1+|m^a(x)|^2}\,\de x+|\mu^s|(A)$;
\item[-] $\S{N-1}$ denotes the unit sphere in $\R N$;
\item[-]  for any $r>0$, $B_r$ denotes the open ball of $\R{N}$ centred at the origin of radius $r$; for any $x\in\R{N}$, $B_r(x) \coloneqq x+ rB$ denotes the open ball centred at $x$ of radius $r$; $Q\coloneqq (-\tfrac12,\tfrac12)^N$ denotes the open unit cube of $\R{N}$ centred at the origin; for any $\eta\in\S{N-1}$, $Q_\eta$ denotes the open unit cube in $\R{N}$ with two faces orthogonal to $\eta$; for any $x\in\R{N}$ and $\delta>0$, $Q(x,\delta)\coloneqq x+\delta Q$ denoted the open cube in $\R{N}$ centered at $x$ with side $\delta$;
\item[-] for any $r>0$ and $U\subset\R{N}$, $U_r\coloneqq\{x\in U:\dist{x}{\partial U}>r\}$ and $U^r\coloneqq
U+B_r$;
\item[-]  $C$ represents a generic positive constant that may change from line to line.
\end{itemize}

\subsection{$SBV$ functions}\label{section:SBV}
We recall some facts on functions of bounded variation and refer the reader to \cite{AFP} for a detailed treatment of this subject. 

Let $V\subseteq\R{N}$ be an open set.
A function $u \in L^1(V; \R d)$ is said to be of \emph{bounded variation}, and we write $u \in BV(V; \R d)$, if its distributional derivative $D u \in \cM (V;\R{d\times N})$, that is, it is a (signed) finite Radon measure. 
The space $BV(V; \R d)$ is a Banach space when endowed with the norm
$\lVert u\rVert_{BV(V;\R{d})} \coloneqq \lVert u\rVert_{L^1(V;\R{d})} + | Du|(V)$.
Since the presence of the total variation $|Du|$ makes this norm too strong for practical applications, it is customary to consider the weak-* convergence in $BV$, which is the appropriate notion for having good compactness properties.
We first need to introduce the notion of weak-*convergence of measures: we say that a sequence of measures $\mu_n\in\cM(V;\R{\ell})$ converges weakly-* to $\mu$ in $\cM(V;\R{\ell})$ 
if
\begin{equation*}
\lim_{n\to\infty}\int_V\varphi(x)\,\de\mu_n(x) = \int_V \varphi(x)\,\de\mu(x) \qquad\text{for every $\varphi\in C_0(V;\R{\ell})$}.
\end{equation*}
With this definition, we say that a sequence $u_n\in BV(V; \R d)$ converges weakly-* to a function $u\in BV(V; \R d)$, in symbols $u_n\wsto u$, if
$$u_n\to u \quad\text{in $L^1(V;\R d)$}\qquad\text{and}\qquad Du_n\wsto Du\quad\text{in $\cM(V;\R{d\times N})$.}$$ 

The Radon-Nikod\'ym Theorem \cite[Theorem~1.28]{AFP} ensures that, for any Radon measure $\mu\in\cM(V;\R{\ell})$, there exists a unique pair of Radon measures $\mu^a$ and $\mu^s$ such that $\mu^a$ is absolutely continuous with respect to the Lebesgue measure $\cL^N$, $\mu^s$ is singular with respect to $\cL^N$, and $\mu=\mu^a+\mu^s$. Moreover, there exists a unique function $m^a\in L^1(V;\R{\ell})$ such that $\mu^a=m^a\cL^N$, so that $\mu=m^a\cL^N+\mu^s$.
The singular part $\mu^s$ of $\mu$ is supported on a set of Lebesgue measure zero.
Since the distributional derivative $Du$ is a measure, it can be split into the sum of two mutually singular measures, which we denote by $D^{a}u$ and $D^{s}u$, the former being absolutely continuous with respect to the Lebesgue measure $\cL^N$, the latter being orthogonal to it. 
By $\nabla u$ we denote the density of $D^{a}u$ with respect to $\mathcal L^N$, so that we can write
$$Du= \nabla u \,\mathcal L^N + D^{s}u.$$
The measure $D^s u$ can be further split into the sum of two contributions, $D^j u$ measuring the discontinuities of $u$ and $D^c u$ measuring the Cantor-like behavior of the distributional derivative. In particular, denoting by $S_u$ the set of points $x\in\Omega$ for which there exist two vectors $a, b\in \R{d}$ and  a unit vector $\nu  \in \S{N-1}$, normal to $S_u$ at $x$,
such that $a\neq b$ and
$$
\lim_{\eps \to 0^+} \frac {1}{\e^N} \int_{\{ y \in x+\eps Q_{\nu} : (y-x)\cdot\nu  > 0 \}} | u(y) - a| \, \de y = 0,
\qquad
\lim_{\eps \to 0^+} \frac {1}{\e^N} \int_{\{ y \in x+\eps Q_{\nu} : (y-x)\cdot\nu  < 0 \}} | u(y) - b| \, \de y = 0, 
$$
the triple $(a,b,\nu)$ is uniquely determined by the two limits above 
up to permutation of $a$ and $b$ and a change of sign of $\nu$ and is denoted by $(u^+ (x),u^- (x),\nu_u (x))$. The set $S_u$ is called the \emph{jump set} of $u$ and it is $(N-1)$-rectifiable.
In conclusion, the distributional derivative $Du$ can be written as the sum of three mutually singular measures as
$$Du= \nabla u \,\mathcal L^N + [u] \otimes \nu_u \,{\mathcal H}^{N-1}\res S_u + D^c u,$$
where $[u](x):= u^+(x) - u^-(x)$ is the jump of $u$ at $x\in S_u$ and $\nu_u(x)$ is the outer unit normal to $S_u$ at $x$. 

The space of \emph{special functions of bounded variation}, $SBV(V; \R d)$ is the space of functions $u \in BV(V; \R d)$ such that $D^cu = 0$; therefore, for each $u\in SBV(V;\R{d})$
\begin{equation}\label{810}
Du = \nabla u \cL^N + [u] \otimes \nu_u \cH^{N-1} \res S_u.
\end{equation}

\subsection{Structured deformations}\label{sec:prel_SD}
In continuum mechanics, structured deformations \cite{DPO1993} provide a substantial description of the multiscale geometry of deformations.
In light of the modern developments of analytical tools for the energetic formulation of mechanical phenomena, structured deformations have been cast in a variational framework in the pioneering work of Choksi and Fonseca \cite{CF1997}.
In their setting, a (first-order) structured deformation of a body, represented here and henceforth by a bounded, connected, open set $\Omega\subset\R{n}$ with Lipschitz boundary, is a pair $(g,G)$, where $g$ represents the macroscopic deformation of $\Omega$ and $G$ represents the contribution at the macrolevel of smooth submacroscopic geometrical changes; in order to allow the macroscopic deformation $g$ to include non-smooth behavior, such as slips and separations, Choksi and Fonseca required that $g\in SBV(\Omega;\R{d})$. 
The matrix-valued field $G\in L^1(\Omega;\R{d\times N})$ captures the contribution of the smooth submacroscopic geometrical changes to the deformation gradient $\nabla g$, so that a relevant object in the theory of structured deformation is the disarrangement tensor $M\coloneqq \nabla g-G$.
Following \cite{CF1997}, we define the set of structured deformations on $\Omega$ as 
$$SD(\Omega; \R{d}\times\R{d\times N})\coloneqq SBV(\Omega; \R{d})\times L^1(\Omega; \R{d\times N}).$$
We endow the space $SD(\Omega; \R{d}\times\R{d\times N})$ with the natural norm induced by the product structure and we introduce the shorthand notation
$ \|(g,G)\|_{SD(\Omega;\R{d}\times\R{d\times N})} \coloneqq \|g\|_{BV(\Omega;\R{d})} +\|G\|_{L^1(\Omega;\R{d\times N})}$, which we are going to denote simply by $\|(g,G)\|_{SD}$ when no domain specification is needed.
The connection between structured deformations and the actual submacroscopic geometrical changes occurring during a deformation is captured in the Approximation Theorem \cite[Theorem~2.12]{CF1997}, \cite[Theorem~1.2]{S2015} (which are counterparts of  \cite[Theorem~5.8]{DPO1993}), stating that for each $(g,G)\in SD(\Omega;\R{d}\times\R{d\times N})$ there exists a sequence $u_n\in SBV(\Omega;\R{d})$ such that, as $n\to\infty$,
\begin{equation}\label{appCF}
u_n \to g\quad\text{in $L^1(\Omega;\R{d})$}\qquad\text{and}\qquad \nabla u_n\wsto G\quad\text{in $\cM(\Omega;\R{d\times N})$}.
\end{equation}
In the formula above, the geometrical process of upscaling from the submacroscopic to the macroscopic level is made precise via the notions of convergence used there.
The approximating functions $u_n$ in \eqref{appCF} are interpreted as a description of both smooth and non-smooth submacroscopic geometrical changes, and we may write
\begin{equation*}
M=\nabla\big(\lim_{n\to\infty} u_n\big)-\lim_{n\to\infty}\nabla u_n.
\end{equation*}
Thus, the disarrangement tensor emerges as a measure of the non-commutativity of the limit operation and taking the absolutely continuous part of the distributional derivative; because of this, it captures the contribution in the limit of the jump discontinuities of the $u_n$'s.
Notice that the approximating sequence $u_n$ in \eqref{appCF} need not be unique.

It is convenient to restate the Approximation Theorem along the lines of \cite[Theorem~1.2]{S2015} to deduce suitable properties of the approximating sequences which we are going to need to prove our main results.
\begin{proposition}[Approximation Theorem]\label{appTHM}
There exists $C>0$ such that for every $(g,G)\in SD(\Omega;\R{d}\times\R{d\times N})$ there exists a sequence $u_n\in SBV(\Omega;\R{d})$ converging to $(g,G)$ according to \eqref{appCF} and such that, for all $n\in\N{}$,
\begin{equation}\label{appEST}
|D u_n|(\Omega)\leq C \|(g,G)\|_{SD(\Omega;\R{d}\times\R{d\times N})}. 
\end{equation}
In particular, 
this implies that, up to a subsequence, 
\begin{equation}\label{centerline}
D^s u_n\wsto (\nabla g-G)\cL^N+D^s g\qquad \text{in $\cM(\Omega;\R{d\times N})$.}
\end{equation}
\end{proposition}
The proof rests on the following two results.
\begin{theorem}[{\cite[Theorem~3]{AL}}]\label{Al}
Let $f \in L^1(\Omega; \R{d{\times} N})$. 
Then there exist $h \in SBV(\Omega; \R d)$ and a Borel function $\beta\colon\Omega\to\R{d{\times} N}$ such that
\begin{equation}\label{817}
Dh = f \,{\cL}^N + \beta \cH^{N-1}\res S_h, \qquad
\int_{S_h\cap\Omega} |\beta(x)| \, \de \cH^{N-1}(x) \leq C_N \lVert f\rVert_{L^1(\Omega; \R{d {\times} N}),}
\end{equation}
where $C_N>0$ is a constant depending only on $N$. \qed
\end{theorem}
\begin{lemma}[{\cite[Lemma~2.9]{CF1997}}]\label{ctap}
Let $h \in BV(\Omega; \R d)$. Then there exist piecewise constant functions $\bar h_n\in SBV(\Omega;\R d)$  such that $\bar h_n \to h$ in $L^1(\Omega; \R d)$ and
\begin{equation}\label{818}
|Dh|(\Omega) = \lim_{n\to \infty}| D\bar h_n|(\Omega) = \lim_{n\to \infty} \int_{S_{\bar h_n}} |[\bar h_n](x)|\; \de\cH^{N-1}(x).
\end{equation}
\end{lemma}
\begin{proof}[Proof of Proposition~\ref{appTHM}]
Let $(g,G)\in SD(\Omega;\R{d}\times\R{d\times N})$ and, by Theorem~\ref{Al} with $f\coloneqq \nabla g-G$, let $h\in SBV(\Omega;\R{d})$ be such that $\nabla h=\nabla g-G$.
Furthermore, let $\bar h_n\in SBV(\Omega;\R{d})$ be a sequence of piecewise constant functions approximating $h$, as per Lemma~\ref{ctap}.
Then, the sequence of functions 
\begin{equation}
u_n\coloneqq g+\bar h_n-h
\end{equation}
is easily seen to approximate $(g,G)$ in the sense of \eqref{appCF}. 
In fact, $ u_n \to g$ in $L^1$ and $\nabla u_n(x) = G(x)$ for $\cL^N$-a.e.~$x\in\Omega$.
Invoking the triangle inequality, the inequality in \eqref{817}, and \eqref{818}, we obtain for $C=3(1+C_N)$
\begin{equation}\label{819}
|Du_n|(\Omega)\leq C\|(g,G)\|_{SD(\Omega;\R{d}\times\R{d\times N})},\qquad\text{for all $n\in\N{}$ sufficiently large,}
\end{equation}
so that \eqref{appEST} is proved for a suitable ``tail'' of the sequence $u_n$. 
The convergence of $u_n\to g$ in $L^1$ implies that $Du_n$ converges to $Dg$ in the sense of distributions.
The uniform bound \eqref{819} ensures the existence of a (not relabeled) weakly-* converging subsequence such that $Du_n\wsto Dg$ in $\cM(\Omega;\R{d\times N})$, so that, since $\nabla u_n\wsto G$ in $\cM(\Omega;\R{d\times N})$, we have
$$
D^s u_n\wsto (\nabla g-G)\cL^N+D^s g\qquad \text{in $\cM(\Omega;\R{d\times N})$,}
$$
which is \eqref{centerline}. The proof is concluded.
\end{proof}

We now define the set of admissible sequences for the limit problems that we consider. Given $(g,G)\in SD(\Omega;\R{d}\times\R{d\times N})$, we let
\begin{equation}\label{Addef}
\Ad(g,G)\coloneqq\{u_n\in SBV(\Omega;\R{d}): \text{$u_n$ converges to $(g,G)$ in the sense of \eqref{appCF} and \eqref{centerline} holds}\}.\end{equation}
Lemma~\ref{appTHM} guarantees that, for every $(g,G)\in SD(\Omega;\R{d}\times\R{d\times N})$, the set $\Ad(g,G)$ is not empty.


\subsection{Measure Theory}\label{sec:MT}
We collect here some basic definitions and results from measure theory that will be used throughout the paper. In particular, we present with more details the notions of weak-* and $\langle\cdot\rangle$-strict convergences for measures, we define two classes of admissible non-local energy densities $\Psi$, and we conclude by stating two Reshetnyak-type continuity theorems. 




In addition to the notion of weak-* convergence of measures introduced in Section~\ref{section:SBV} (and repeated in point (i) of Definition~\ref{1100}), we give further definitions of convergence, which will be relevant in the sequel. We point out that it is crucial to distinguish if the domain is an open set $U$ or a closed one $\overline U$.
\begin{definition}\label{1100}
Let $\mu_n,\mu\in\cM(U;\R{\ell})$ (or in $\cM(\overline U;\R{\ell})$).
\begin{itemize}
\item[(i)] 
We say that $\mu_n$ converges \emph{weakly-*} to $\mu$ in $\cM(U;\R{\ell})$ 
if
\begin{equation*}
\lim_{n\to\infty}\int_U \varphi(x)\,\de\mu_n(x) = \int_U \varphi(x)\,\de\mu(x) \qquad\text{for every $\varphi\in C_0(U;\R{\ell})$.}
\end{equation*}
\item[(ii)] We say that $\mu_n$ converges \emph{locally weakly-*} to $\mu$ in $\cM(U;\R{\ell})$ if
\begin{equation*}
\lim_{n\to\infty}\int_U \varphi(x)\,\de\mu_n(x) = \int_U \varphi(x)\,\de\mu(x) \qquad\text{for every $\varphi\in C_c(U;\R{\ell})$}.
\end{equation*}
\end{itemize}
(Notice that (i) and (ii) coincide if $U$ is bounded.)
\begin{itemize}
\item[(iii)] 
We say that $\mu_n$ converges \emph{weakly-*} to $\mu$ in $\cM(\overline U;\R{\ell})$ if
\begin{equation*}
\lim_{n\to\infty}\int_{\overline U} \varphi(x)\,\de\mu_n(x) = \int_{\overline U} \varphi(x)\,\de\mu(x) \qquad\text{for every $\varphi\in C_{\blu{0}}(\overline U;\R{\ell})$.}
\end{equation*}
\end{itemize}
The weak-* convergence of $\mu_n$ to $\mu$ is denoted by the symbol $\mu_n\wsto\mu$.
\begin{itemize}
\item[(iv)] We say that $\mu_n$ \emph{converges $\langle\cdot\rangle$-strictly} to $\mu$ in $\cM(U;\R{\ell})$ (or in $\cM(\overline U;\R{\ell})$) if $\mu_n\wsto\mu$ and in addition, $\langle\mu_n\rangle(U)\to\langle\mu\rangle(U)$
(or $\langle\mu_n\rangle(\overline U)\to\langle\mu\rangle(\overline U)$), where, for every $A\in \cB(U)$ (or $A\in \cB(\overline U)$),
\begin{equation*}\label{1017}
\langle\mu\rangle(A)\coloneqq \int_A \sqrt{1+|m^a(x)|^2}\,\de x+|\mu^s|(A).
\end{equation*}
\end{itemize}
\end{definition}
The next proposition establishes semicontinuity and continuity results results with respect to the weak-* convergence of measures.
\begin{proposition}[{\cite[Proposition~1.62]{AFP}}]\label{AFPb}
Let $\mu_n\in\cM(U;\R{\ell})$ be a sequence of bounded Radon measures locally weakly-* converging to $\mu$. 
\begin{itemize}
\item[(a)] If the measures $\mu_h$ are positive, then for every lower semicontinuous function $u\colon U\to[0,+\infty]$
$$\int_U u(x)\,\de\mu(x)\leq\liminf_{h\to\infty} \int_U u(x)\,\de \mu_h(x)$$
and for every upper semicontinuous function $v\colon U\to[0,+\infty)$ with compact support
$$\int_U v(x)\,\de\mu(x)\geq\limsup_{h\to\infty} \int_U v(x)\,\de \mu_h(x).$$
\item[(b)] If $|\mu_n|$ locally weakly-* converges to $\Lambda$, then $\Lambda \geq |\mu|$.
Moreover, if $K$ is a relatively compact $\mu$-measurable subset of $U$ such that $\Lambda(\partial K) = 0$, then $ \mu_n (K) \to \mu(K)$ as $n \to \infty$.
More generally,
$$
\int_{U} u(x) \, \de \mu(x) = \lim_{n\to \infty} \int_{U} u(x) \, \de \mu_n(x),
$$
for every bounded Borel function $u\colon U \to \R{}$ with compact support, such that the set of discontinuity points is $\Lambda$-negligible. \qed
\end{itemize}
\end{proposition}
Recall that given two functions $f, g $ defined in $\R{N}$, the \emph{convolution} is a commutative and associative smoothing operation defined by
$$\big (g \ast f \big)(x) \coloneqq\int_{\mathbb R^{N}} f(x-y)g(y)d (y),$$
whenever the integral makes sense. One of the main features of the convolution operation is that the function $f*g$ is as regular as the most regular between $f$ and $g$.

Following \cite[Definition 2.1]{AFP}, we define for every $\mu \in \mathcal{M} (U;\mathbb R^\ell)$ and  every continuous function $f\colon \R{n}\to \mathbb R$ the \emph{convolution} 
of $\mu$ and $f$ (denoted by $\mu \ast f$)
$$(\mu \ast f)(x)\coloneqq\int_U f(x-y)d \mu (y).$$
In this case, also, the regularity of $\mu*f$ is the higher between that of $\mu$ and that of $f$. In particular, since $f$ is continuous, so is $\mu*f$, so that this turns out to smoothen the measure $\mu$. Through the convolution, the singularities of $\mu$ are smeared, so that $\mu*f$ can be used as a density of the Lebesgue measure to generate a measure which is absolutely continuous with respect to $\cL^N$. The convergence results that we are going to mention establish the behavior of convolution under weak-* convergence.

Following \cite[page~41]{AFP}, if $\mu_h\in \mathcal M(U;\mathbb R^\ell)$ is a sequence locally weakly-$^*$ converging in $U$ to $\mu$ and $f \in C_c^\infty(\mathbb R^N)$, 
then 
\begin{align}\label{cuconv}
	\mu_h \ast f \to \mu \ast f\quad \text{uniformly on the compact sets of $\mathbb R^N$.}
\end{align}

We now introduce the convolution kernels that will be responsible for the averaging process in the non-local term of the energy. They play the role of the division by $2r$ in the functional \eqref{non-local term}. Let
\begin{subequations}\label{1014}
\begin{equation}\label{1014a}
\alpha_r(x)\coloneqq \frac1{r^N}\alpha\Big(\frac xr\Big),
\end{equation}
where
\begin{equation}\label{1014b}
\alpha \in C_c^\infty(\R{N}), \qquad \spt \alpha \subset B_1,
\qquad \text{and}
\quad \int_{B_1} \alpha(x)\,\de x=1.
\end{equation}
\end{subequations}
It is easily verified that, for every $x \in \R{N}$, the convolution
\begin{equation}\label{convalpha}
\mu \ast \alpha_r(x)=\int_U \alpha_r(x-y)d \mu (y)= r^{-N}\int_{B_r(x)\cap U} \alpha\left(\frac{x-y}{r}\right) d \mu(y).
\end{equation}
is well defined and it is a regular function. The convolution of a measure with a kernel of the type in \eqref{1014} has good convergence properties with respect to the weak-* and $\langle\cdot\rangle$-strict convergences of measures.
\begin{theorem}[{\cite[Theorem 2.2]{AFP}, \cite[Proposition 2.22]{BCMS2013}}]\label{815}
Let $\mu\in\cM(U;\R{\ell})$ and, for $r>0$, let $\alpha_r$ be as in \eqref{1014}. 
Define the measures $\mu_r\coloneqq (\mu*\alpha_r)\cL^N$.
Then, as $r\to0^+$,
\begin{itemize}
\item[(i)] $\mu_r$ locally weakly-* converge to $\mu$ and for every $A\in\cB(U_r)$
\begin{equation}\label{lT816}
\int_{A} |\mu*\alpha_r|(x)\,\de x\leq|\mu|(A^r)
\end{equation}
and the measures $|\mu_r|$ locally weakly-* converge in $U$ to $|\mu|$;
\item[(ii)] if  $ |\mu|(\partial U) =0$, there also holds $\langle\mu_r\rangle(U) \to \langle\mu\rangle(U)$. \qed
\end{itemize}
\end{theorem}
We introduced the convolution kernels $\alpha_r$ of \eqref{1014} which we will use to describe the averaging of the jumps in the non-local term of the energy. We now introduce and describe two classes of continuous functions from which we are going to take the energy density of the non-local energy.
We consider a continuous function $\Psi\colon\Omega\times\R{d\times N}\to[0,+\infty)$ belonging to the class (E) or (L), which are characterized as follows:
\begin{itemize} 
\item[(E)] $\Psi$ can be extended to a function (still denoted by $\Psi$) belonging to $C(\overline\Omega\times\R{d\times N})$ with the property that  $\lim_{t\to+\infty} \Psi(x,t\xi)/t$ exists uniformly in $x\in\overline\Omega$ and $\xi$ with $|\xi|=1$. 
(Such functions $\Psi$ form the class $\mathbf{E}(\Omega\times\R{d\times N})$ defined in \cite[Section~2.4]{KR}.)
In particular, this entails that
\begin{itemize} 
\item[(i)] $\Psi$ has at most linear growth at infinity with respect to the second variable, namely there exists $C_\Psi>0$ such that
\begin{equation}\label{S100}
|\Psi(x,\xi)|\leq C_\Psi(1+|\xi|)\qquad\text{for all $x\in\Omega$ and $\xi\in\R{d\times N}$;}
\end{equation}
\item[(ii)] for all $x\in \overline\Omega$ and $\xi\in\R{d\times N}$ there exists the limit
\begin{equation}\label{S101}
\lim_{\substack{x'\to x \\\xi'\to \xi \\ t\to+\infty}} \frac{\Psi(x',t\xi')}{t}.
\end{equation}
\end{itemize}
\item[(L)] 
\begin{itemize}
\item[(i)] $\Psi$ is Lipschitz with respect to the second variable, \emph{i.e.}, there exists $L_\Psi>0$ such that 
\begin{equation}\label{S102}
|\Psi(x,\xi)-\Psi(x,\xi')|\leq L_\Psi|\xi-\xi'|,\qquad\text{for all $x\in\Omega$ and $\xi,\xi'\in\R{d\times N}$;}
\end{equation}
\item[(ii)] there exists a continuous function $\omega\colon[0,+\infty)\to[0,+\infty)$, with $\omega(s)\to0^+$ as $s\to0^+$, such that
\begin{equation}\label{S103}
|\Psi(x,\xi)-\Psi(x',\xi)|\leq\omega(|x-x'|)(1+|\xi|),\qquad\text{for all $x,x'\in\Omega$, $\xi\in\R{d\times N}$.}
\end{equation}
\end{itemize}
Notice that, by fixing $\xi'\in\R{d\times N}$, \eqref{S102} implies that there exists $C_\Psi>0$ such that \eqref{S100} holds.
\end{itemize}

\begin{remark}\label{diverse}
The two classes (E) and (L) have a non-empty intersection, but also a non-trivial symmetric difference.
\\
An example of a function which belongs to (E) but not to (L), with $N=d=1$ (and hence $\ell=1$), is the function $\Psi\colon\R{}\to\R{}$ defined by $\Psi(\xi)=\sqrt{1-\xi^2}$ for $\xi\in[-1,1]$ and extended by periodicity. The limit in \eqref{S101} exists and equals $0$, but $\Psi$ is not Lipschitz.
\\
An example of a function which belongs to (L) but not to (E), again with $N=d=\ell=1$, is given by $\Psi\colon\R{}\to\R{}$ defined in terms of the sequence $\{\xi_n\}_{n=1}^{\infty}$, defined recursively by $\xi_1=1$ and $\xi_{n+1}=2n\xi_n$, for $n\in\N{}\setminus\{0\}$, and such that
$$\Psi(\xi)=
\begin{cases}
0, & 0\leq \xi \leq1, \\
\xi-\xi_n, & \xi_n\leq \xi \leq \displaystyle\frac{\xi_n+\xi_{n+1}}{2}, \quad n\in\N{}\setminus\{0\}, \\
\xi_{n+1}-\xi, & \displaystyle\frac{\xi_n+\xi_{n+1}}{2}\leq \xi \leq \xi_{n+1}, \quad n\in\N{}\setminus\{0\}.
\end{cases}
$$
Then $\Psi(\xi_n)/\xi_n=0$ for all $n\in\N{}\setminus\{0\}$ and 
$$\frac{\Psi\Big(\displaystyle\frac{\xi_n+\xi_{n+1}}{2}\Big)}{\displaystyle\frac{\xi_n+\xi_{n+1}}{2}}=\frac{\xi_{n+1}-\xi_n}{\xi_{n+1}+\xi_n}=\frac{2n-1}{2n+1},\qquad\text{for $n\in\N{}\setminus\{0\}$}.$$
Consequently,
$$\lim_{n\to\infty} \frac{\Psi(\xi_n)}{\xi_n}=0<1=\lim_{n\to\infty}\frac{\Psi\Big(\displaystyle\frac{\xi_n+\xi_{n+1}}{2}\Big)}{\displaystyle\frac{\xi_n+\xi_{n+1}}{2}},$$
so that 
\eqref{S101} does not hold. 
\qed
\end{remark}

We now introduce functionals defined on measures, which will be useful for our mathematical treatment of the problem we study (see Section~\ref{section:mathform}).
Given $\mu=m^a\cL^N+\mu^s\in\cM(U;\R{\ell})$ and $\Phi\colon U\times \R{\ell}\to[0,+\infty)$ continuous, let 
\begin{equation}\label{959}
\sI(\mu) \coloneqq \int_{U} \Phi (x, m^a(x))\,\de x +\int_{U} \Phi^\infty \bigg(x,\frac{\de \mu^s}{\de |\mu^s|}(x)\bigg)\,\de |\mu^s|(x),
\end{equation}
where $\Phi^\infty$ is the \emph{recession function} of $\Phi$ at infinity, defined by
\begin{equation}\label{defrecession2}
\Phi^\infty(x,\xi)\coloneqq \limsup_{\substack{x'\to x \\ \xi' \to \xi\\ t\to +\infty}}\frac{\Phi(x',t \xi')}{t}
\end{equation}
for every $x\in\overline U$, $\xi\in\S{\ell-1}$ and extended to $\R{\ell}$ by positive $1$-homogeneity. 
The following two results will be useful in Section~\ref{rto0}. 
Notice that if 
$\Phi$ belongs to the class (E) then $\Phi^\infty$ is a limit, namely
\begin{equation}\label{defrecession}
\Phi^\infty(x,\xi)= \lim_{\substack{x'\to x \\ \xi' \to \xi\\ t\to +\infty}}\frac{\Phi(x',t \xi')}{t}
\end{equation}
for every $x\in\overline U$, $\xi\in\S{\ell-1}$ and extended to $\R{\ell}$ by positive $1$-homogeneity. We point out that (i) continuous and positively $1$-homogeneous functions and (ii) convex functions with linear growth are two classes of functions belonging to 
(E) (see \cite{KR2}).
We refer the reader to \cite[Section~2.4]{KR} for a detailed description of the class (E), where it is indicated by the symbol $\mathbf{E}(U\times\R{\ell})$.

\begin{theorem}[{Reshetnyak upper-semicontinuity theorem, \cite[Corollary~2.11]{BCMS2013}}]\label{ReshetnyaktheoremBCMS}
Let $\mu_n,\mu
\in\cM(U;\R{\ell})$ 
be such that $\mu_n$ $\langle\cdot\rangle$-strictly converges to $\mu$.
Let $\Phi\colon U\times\R{\ell}\to[0,+\infty)$ be a continuous function
with linear growth at infinity (see \eqref{S100}).
Then 
the functional $\sI$ defined in \eqref{959} is upper semicontinuous, namely
$$\sI(\mu)\geq\limsup_{n\to\infty} \sI(\mu_n).$$
\end{theorem}
\begin{theorem}[{Reshetnyak continuity theorem, \cite[Theorem~4]{KR}}]\label{ReshetnyaktheoremKR}
Let $\mu_n, \mu
\in\cM(\overline U;\R{\ell})$ 
be such that $\mu_n$ $\langle\cdot\rangle$-strictly converges to $\mu$
and let $\Phi$ belong to the class (E). 
Then $\Phi^\infty$ is given by \eqref{defrecession} and 
$$\overline\sI(\mu_n ) \to \overline\sI(\mu),\qquad \text{as $n\to\infty$},$$
where
$$\overline\sI(\mu)\coloneqq \int_{U} \Phi (x, m^a(x))\,\de x +\int_{\overline U} \Phi^\infty \bigg(x,\frac{\de \mu^s}{\de |\mu^s|}(x)\bigg)\,\de |\mu^s|(x),$$
and analogously for $\overline\sI(\mu_n)$.
\end{theorem}

We conclude this subsection by proving the following property for functions $\Phi$ belonging to the class (L).
\begin{lemma}\label{S105}
Let $\Phi\colon U\times\R\ell\to[0,+\infty)$ belong to the class (L).
Then the recession function $\Phi^\infty$ defined in \eqref{defrecession2} can be computed as
\begin{equation}\label{defrecession3}
\Phi^\infty(x,\xi)=\limsup_{t\to+\infty} \frac{\Phi(x,t\xi)}{t},\qquad\text{for all $x\in U$, $\xi\in\R{\ell}$.}
\end{equation}
\end{lemma}
\begin{proof}
Fix $x\in U$ and $\xi\in\R{\ell}$.
The inequality $\Phi^\infty(x,\xi)\geq \limsup_{t\to+\infty} t^{-1}\Phi(x,t\xi)$ is obvious from the definition of $\limsup$.
The proof of the converse inequality is a matter of a computation, using the subadditivity of the $\limsup$ and keeping \eqref{S102} and \eqref{S103} in mind.
\end{proof}

\section{Theorems on upscaling and localization of non-local energies}\label{section:mathform}
In this section we introduce a higher-dimensional analogue of the non-local functional \eqref{non-local term}. 
In what follows, $\Omega\subset\R{N}$ denotes a bounded, connected, open set with Lipschitz boundary and we consider deformations of the body $\Omega$ taking values in $\R{d}$. 
Therefore the deformation gradient is a $\R{d\times N}$-valued field defined on $\Omega$. As it is customary when assigning an energy to a structured deformation, we start from an initial energy defined for a classical deformation $u\colon\Omega\to\R{d}$ of the body and define the energy of a structured deformation $(g,G)$ as the energetically most economical way to approximate $(g,G)$ by means of classical deformations $u_n$ which converge to $(g,G)$ according to \eqref{appCF}.
The theory that we are going to introduce here and develop in Section~\ref{section:new3} is for a non-local energy inspired by \eqref{non-local term}, whereas in Section~\ref{together} we pair the non-local functional with a local one of the type treated in \cite{CF1997} and present results on the upscaling and spatial localization of an energy featuring both local and non-local terms. 
It is for this coupling result that we are going to need the sequences $u_n$ to converge to $(g,G)$ in the sense of \eqref{appCF} and also satisfy \eqref{centerline}. We are going to describe this process in detail.

The main issues that Choksi and Fonseca \cite{CF1997} addressed were the assignment of an energy to a structured deformation and the establishment of an integral representation for that energy.
They took an initial energy $E_L\colon SBV(\Omega;\R{d})\to[0,+\infty)$ featuring a bulk energy density $W\colon \R{d\times N}\to[0,+\infty)$ and an interfacial energy density $\psi\colon\R{d}\times\S{N-1}\to[0,+\infty)$ in the form
\begin{equation}\label{1002}
E_L(u)\coloneqq \int_\Omega W(\nabla u(x))\,\de x+\int_{\Omega\cap S_u} \psi([u](x),\nu_u(x))\,\de\cH^{N-1}(x),
\end{equation}
where $\de x$ and $\de \cH^{N-1}(x)$ denote integration with respect to the $N$-dimensional Lebesgue measure $\cL^N$ and the $(N-1)$-dimensional Hausdorff measures, respectively, $S_u$ is the jump set of $u$, $[u](x)$ is the jump of $u$ at $x\in S_u$, and $\nu_u(x)$ is the outer unit normal at $x\in S_u$.
Because of the non-uniqueness of the approximating sequence $u_n$ in Proposition~\ref{appTHM}, the energy $I_L(g,G)$ for a given structured deformation is defined as the most economical way, in terms of energies $E_L(u_n)$ in \eqref{1002}, to reach $(g,G)$.
From the mathematical point of view, this corresponds to a relaxation/upscaling procedure, namely
\begin{equation}\label{1003}
\begin{split}
I_L(g,G)\coloneqq \inf_{\{u_n\}\subset SBV(\Omega;\R{d})}\Big\{\liminf_{n\to\infty} E_L(u_n): &\; \text{$u_n$ converges to $(g,G)$ as in \eqref{appCF}} \\
&\; \text{and $\sup_n 
\|\nabla u_n\|_{L^p(\Omega;\R{d\times N})} 
<\infty$}\Big\}
\end{split}
\end{equation}
where $p\geq1$.
The representation theorems \cite[Theorems~2.16 and~2.17]{CF1997} state that, under suitable hypotheses on $W$, $\psi$, and $G$ depending upon $p$, there exist a relaxed/upscaled bulk energy density $H\colon \R{d\times N}\times \R{d\times N}\to[0,+\infty)$ and a  relaxed/upscaled interfacial energy density $h\colon\R{d}\times\S{N-1}\to[0,+\infty)$ such that 
\begin{equation}\label{1004}
I_L(g,G)=\int_\Omega H(\nabla g(x),G(x))\,\de x+\int_{\Omega\cap S_g} h([g](x),\nu_g(x))\,\de\cH^{N-1}(x).
\end{equation}
We refer the reader to Section~\ref{ChoksiFonseca} for the integral representation theorem providing \eqref{1004}; the particular hypotheses on $W$, $\psi$, and $G$ that depend upon $p$ will not play a role until then.
As a matter of fact, we will present a more general version where we allow the initial bulk and surface energy densities $W$ and $\psi$ to depend on the space variable $x$.
The one-dimensional procedure of \cite{DPO2000} inspired by that in \cite{CF1997} was carried out for the notion of structured deformations introduced in \cite{DPO1993}; there the initial energy \eqref{1002} had the form
\begin{equation}\label{1007}
E_L(u)=\int_0^1 W(\nabla u(x))\,\de x+\sum_{z\in S_u} \psi([u](z))
\end{equation}
and the resulting integral representation \eqref{1004} was shown to be
\begin{equation*}
I_L(g,G)=\int_0^1 \big(W(G(x))+\zeta(\nabla g(x)-G(x))\big)\,\de x+\sum_{z\in S_g} \psi([g](z)),
\end{equation*}
where $\zeta\coloneqq \liminf_{\zeta\to0^+} \psi(\zeta)/\zeta$.
In this example, the contribution to the relaxed bulk energy density $H$ of the initial interfacial energy density $\psi$ has a special character: as the definition of $\zeta$ shows, only arbitrarily small jumps influence the relaxed bulk response, which, in turn, is linear in the disarrangement tensor $M$.
In \cite{CDPFO1999} and subsequently in \cite{DO2000,DO2002} a periodic dependence upon $M$ was shown to account for yielding, hysteresis, and hardening in single crystals undergoing two-level shears.
Therefore, to include such significant non-linear effects, the choice \eqref{1007} of initial energy must be modified.
The proposal in \cite{DPO2000} toward capturing a non-linear dependence on $M$ was to modify the initial energy \eqref{1007} by adding the non-local term \eqref{non-local term} to form the initial energy
\begin{equation}\label{1006}
F^r(u)\coloneqq\int_0^1 W(\nabla u(x))\,\de x+\sum_{z\in S_u} \psi([u](z))+\int_0^{1} \Psi \bigg(\sum_{z\in S_u\cap (x-r,x+r)} \frac{[u](z)}{2r}\bigg)\de x,
\end{equation}
where the added, non-local term includes the bounded and uniformly continuous bulk energy density $\Psi\colon [0,+\infty)\to[0,+\infty)$ accounting for the average of the jumps within each interval of radius $r$. 
Passing to structured deformations in \eqref{1006} 
and then taking the limit as $r\to0^+$ yields (see \cite[Proposition~2.3 and (2.21)]{DPO2000})
\begin{equation}\label{1008}
J(g,G) =\int_0^1 \big(W(G(x))+\zeta(\nabla g(x)-G(x))\big)\,\de x+\sum_{z\in S_g} \psi([g](z))+ \int_0^1 \Psi (\nabla g(x)-G(x))\de x,
\end{equation}
where a second, possibly non-linear, dependence on the disarrangements appears through the density $\Psi$ in the last integral above.

\smallskip

A principal goal of this paper is to show that an analogous procedure that achieves in one dimension the energy in \eqref{1008} can be carried out in higher dimensions in the $SBV$ framework of \cite{CF1997}, by adding to the energy $E_L$ in \eqref{1002} a term similar to the last term on the right-hand side of \eqref{1006}.
For a continuous function $\Psi\colon \Omega\times\R{d\times N}\to[0,+\infty)$ in (E) or (L), set $\Omega_r\coloneqq\{x\in\Omega:\dist{x}{\partial\Omega}>r\}$, and for $u\in SBV(\Omega;\R{d})$, define the averaged interfacial energy
\begin{equation}\label{1009}
E^{\alpha_r}(u)\coloneqq \int_{\Omega_r} \Psi\big(x,(D^su*\alpha_r)(x)\big)\,\de x.
\end{equation}
Notice that we have introduced an explicit dependence on $x$ in the non-local energy density $\Psi$ in \eqref{1009}.
The need for such a dependence is motivated by  applications to yielding, hysteresis, and crystal plasticity that we will discuss in Sections~\ref{sect7} and~\ref{plasticity}.
Putting \eqref{convalpha} and \eqref{1009} together and using the structure theorem for the derivative of $SBV$ functions (see formula \eqref{810}), 
yields the following form for the averaged interfacial energy $E^{\alpha_r}$
\begin{equation}\label{1011}
E^{\alpha_r}(u)= \int_{\Omega_r} \Psi\bigg(x,\int_{B_r(x)\cap S_u} \alpha_r(x-y)[u](y)\otimes\nu_u(y)\,\de\cH^{N-1}(y)\bigg)\de x.
\end{equation}
We note that in the expression above the non-local character of the averaged interfacial energy emerges through the appearance of two iterated integrations, the inner surface integral with respect to $\cH^{N-1}$ and the outer volume integral with respect to $\cL^N$.
Our main aim in this paper is to study the behavior of energy \eqref{1011} under upscaling, i.e., as the field $u$ approaches a target structured deformation $(g,G)$, followed by spatial localization, i.e., as $r$ approaches $0$.
The first contribution we obtain is the following 
result concerning the upscaling of the initial energy $E^{\alpha_r}(u_n)$.
\begin{theorem}\label{mainrho}
Let $\Omega\subset\R{N}$ be a bounded Lipschitz domain, $\Psi\colon \Omega\times\R{d\times N} \to [0,+\infty)$ be a continuous function, 
for $r>0$, let $\alpha_r$ be as in \eqref{1014}, and let $E^{\alpha_r}$ be as in \eqref{1011}.
Then for every $(g,G)\in SD(\Omega;\R{d}\times\R{d\times N})$ 
and for every admissible sequence $u_n\in\Ad(g,G)$ (see \eqref{Addef})
\begin{equation}\label{864}
\begin{split}
\lim_{n\to\infty} E^{\alpha_r}(u_n)= &\; \int_{\Omega_r} \Psi\bigg(x, \int_{B_r(x)}  \alpha_r(y-x)(\nabla g-G)(y)\,\de y  \\
& \phantom{ \int_{\Omega_r} \Psi\bigg(x,}
+ \int_{B_r(x)\cap S_g}  \alpha_r(y-x)[g](y)\otimes \nu_g(y) \, \de \cH^{N-1}(y)\bigg)\de x \\ 
 \eqqcolon & \; I^{\alpha_r}(g,G;\Omega_r).
\end{split}
\end{equation}
\end{theorem}
After proving Theorem~\ref{mainrho} 
we deduce an explicit formula for 
\begin{equation}\label{1013}
I(g,G)\coloneqq \lim_{r\to0^+} \lim_{n\to\infty} E^{\alpha_r}(u_n)=\lim_{r\to0^+}I^{\alpha_r}(g,G;\Omega_r),
\end{equation}
where $I(g,G)$ represents the spatial localization of the upscaled energy $I^{\alpha_r}(g,G;\Omega_r)$.
In our main result, we obtain an explicit formula for $I(g,G)$ via an extension of $(g,G)\in SD(\Omega;\R{d}\times\R{d\times N})$ to $(\bar g,\bar G)\in BV(\R{N};\R{d})\times L^1(\R{N};\R{d\times N})$ such that $|D^s\bar g|(\partial\Omega)=0$ (see \cite{Gagliardo,Gerhardt}).
This extension permits us to add a term to $I^{\alpha_r}(g,G;\Omega_r)$ such that the resulting sum $\overline I^{\alpha_r}(\bar g,\bar G;\Omega)$ is an integral over the fixed domain $\Omega$ whose limit can be studied via the Reshetnyak continuity-type Theorems~\ref{ReshetnyaktheoremBCMS} and~\ref{ReshetnyaktheoremKR}, and the resulting explicit formula turns out not to depend on the particular choice of the extension. 
Accordingly, we restrict our attention to functions $\Psi$ with at most linear growth at infinity, belonging to the classes $(E)$ or $(L)$ described in Subsection~\ref{sec:MT}.
We are now in a position to state our result concerning the limit \eqref{1013}, the spatial localization of the upscaled energy $I^{\alpha_r}(g,G;\Omega_r)$ in \eqref{864}.
\begin{theorem}\label{final1E}
Let $\Omega\subset\R{N}$ be a bounded Lipschitz domain, let $\Psi\colon \Omega\times \R{d\times N} \to [0,+\infty)$ be a continuous function belonging to (E) or (L), and let $\alpha_r$ be as in \eqref{1014}.
Then 
for any $(g,G)\in SD(\Omega;\R{d}\times\R{d\times N})$ 
the limiting energy $I(g,G)$ in \eqref{1013} is given by
\begin{equation}\label{1146}
I(g,G) = \int_\Omega \Psi\big(x,\nabla g(x) - G(x)\big)\,\de x+\int_{\Omega\cap S_g}\Psi^{\infty}\Big(x,\frac{\de D^s g}{\de|D^s g|}(x)\Big)\,\de|D^s g|(x),
\end{equation}
with $\Psi^{\infty}$ defined by \eqref{defrecession2}.
%
\end{theorem}
\begin{remark}\label{tuttiuguali}
We observe the following:
\begin{itemize}
\item The recession function $\Psi^\infty$ defined by \eqref{defrecession2} is finite whenever $\Psi$ is in (E) or (L). Notice that it is a limit if $\Psi$ is in (E), see \eqref{S101}.
\item In Theorem~\ref{final1E}, the resulting bulk energy density retains the character of the function $\Psi$ that defines the initial non-local energy \eqref{1009}.
Moreover, we observe that since 
$\Psi^\infty$ vanishes in the case of sublinear growth at infinity formula \eqref{1146} reduces to
\begin{equation}\label{624}
I(g,G) = \int_\Omega \Psi\big(x,\nabla g(x) - G(x)\big)\,\de x
\end{equation}
when $\Psi$ has sublinear growth. \qed
\end{itemize}
\end{remark}
It is now natural to consider an initial energy that combines both a local contribution, described by the functional $E_L$ in \eqref{1002}, and a non-local one, described by the functional $E^{\alpha_r}$ in \eqref{1009}. 
We now focus our attention on the relaxation/upscaling, in the context of \cite{CF1997}, of the energy functional
\begin{equation}\label{1036}
F^{\alpha_r}(u)\coloneqq E_L(u)+E^{\alpha_r}(u),
\end{equation}
namely, we consider 
\begin{equation}\label{1037}
\begin{split}
J^{\alpha_r}(g,G)\coloneqq \inf_{\{u_n\}\subset SBV(\Omega;\R{d})}\Big\{\liminf_{n\to\infty} F^{\alpha_r}(u_n) : &\; 
\text{$u_n$ converges to $(g,G)$ according to \eqref{appCF}} \\
& \;\text{and
$\sup_n 
\|\nabla u_n\|_{L^p(\Omega;\R{d\times N})} 
<\infty$}\Big\}
\end{split}
\end{equation}
where, as in \eqref{1003}, $p\geq1$.
We will prove in Theorem~\ref{main2} that the relaxation/upscaling of the sum $F^{\alpha_r}$ in \eqref{1036} is the sum of the upscaling $I^{\alpha_r}$ in \eqref{864} of $E^{\alpha_r}$ and the relaxation/upscaling $I_L$ in \eqref{1003} of $E_L$:
\begin{equation}\label{1038}
J^{\alpha_r}(g,G)=I_L(g,G)+I^{\alpha_r}(g,G;\Omega_r),
\end{equation}
so that, defining
\begin{equation}\label{1039}
J(g,G)\coloneqq\lim_{r\to0^+} J^{\alpha_r}(g,G)
\end{equation}
and keeping \eqref{1013} in mind, we obtain
\begin{equation}\label{1040}
J(g,G)=I_L(g,G)+I(g,G).
\end{equation}
Eventually, from \eqref{1004} and \eqref{1146} the energy $J(g,G)$ has the explicit expression (see \eqref{812} in Corollary~\ref{finalfinal})
\begin{equation}\label{energyJ}
\begin{split}
J(g,G) = & \int_\Omega H(\nabla g(x),G(x))\,\de x+\int_{\Omega\cap S_g} h([g](x),\nu_g(x))\,\de\cH^{N-1}(x) \\
& + \int_\Omega \Psi\big(x,\nabla g(x) - G(x)\big)\,\de x+\int_{\Omega\cap S_g}\Psi^{\infty}\Big(x,\frac{\de D^s g}{\de|D^s g|}(x)\Big)\,\de|D^s g|(x).
\end{split}
\end{equation}
The formula above couples together the contributions to the total energy $J(g,G)$ coming from the local part $I_L(g,G)$ and from the limit $I(g,G)$ of the non-local energy. The contribution of the singularities of $g$ enters the expression of $J(g,G)$ both through the surface term of $I_L(g,G)$ and through the surface term of $I(g,G)$, via the 
function $\Psi^\infty$, thus retaining the linear character at infinity of $\Psi$. The effect of $\Psi$ on the disarrangements is encoded in the bulk term of $I(g,G)$.
Corollary~\ref{finalfinal} provides the explicit representation \eqref{energyJ} and shows that the nonlinearities introduced in the microlevel energy $E^{\alpha_r}$ through $\Psi$ persist under the two operations of relaxation/upscaling and spatial localization. 

\section{Proofs of Theorems~\ref{mainrho} and~\ref{final1E}}\label{section:new3}
In this section we perform the two limiting operations described in the Introduction. We first upscale the non-local energy $E^{\alpha_r}$ defined in \eqref{1009}, that is, we prove Theorem~\ref{mainrho}, and then we spatially localize the energy $I^{\alpha_r}$ defined in \eqref{864}, that is, we prove Theorem~\ref{final1E}.
\subsection{Upscaling of the non-local energy $E^{\alpha_r}$}\label{section:statement of the problem}
This subsection is devoted to the proof of Theorem~\ref{mainrho}.
\begin{proof}[Proof of Theorem~\ref{mainrho}]
Let $\alpha_r$ be as in \eqref{1014} and let $(g,G)\in SD(\Omega;\R{d}\times\R{d\times N})$.
For any sequence $u_n \in \Ad(g,G)$ (see the definition \eqref{Addef}) let $\mu_n,\mu \in\cM(\Omega;\R{d\times N})$ be defined as
\begin{equation}\label{1204}
\mu_n \coloneqq  D^s u_n, \qquad \mu\coloneqq (\nabla g-G)\cL^N+D^s g.
\end{equation}
Then, since $u_n\in\Ad(g,G)$, we have that 
\begin{equation}\label{621}
\mu_n \wsto\mu.
\end{equation}
For $r \in (0,1)$ and for $x \in \Omega_r$ 
define
\begin{equation*}
\begin{split}
f_{r,n}(x) \coloneqq & \int_{\Omega\cap B_r(x)\cap S_{u_n}} \!\!\!\!\!\! \alpha_r(y-x)[u_n](y)\otimes \nu_{u_n}(y) \, \de \cH^{N-1}(y)=(\mu_n*\alpha_r)(x),\\
f_r(x) \coloneqq & \int_{\Omega\cap B_r(x)} \!\!\!\!\!\! \alpha_r(y-x)(\nabla g(y)-G(y))\,\de y+\int_{\Omega\cap B_r(x)\cap S_g} \!\!\!\!\!\! \alpha_r(y-x)[g](y)\otimes \nu_g(y) \, \de \cH^{N-1}(y)=(\mu*\alpha_r)(x).\\
\end{split}
\end{equation*}
By \eqref{cuconv}, for every $x\in\Omega_r=\{x\in \Omega:\dist{x}{\partial \Omega}>r\}$
\begin{equation}\label{pcconv}
\lim_{n\to\infty} f_{r,n}(x)=f_r(x),
\end{equation} 
and
we claim that
\begin{equation}\label{M11}
\lim_{n\to\infty} \int_{\Omega_r} \Psi(x,f_{r,n}(x))\,\de x = \int_{\Omega_r} \Psi (x,f_r(x))\,\de x.
\end{equation}
Indeed, bounds on $\alpha_r$, which follow from \eqref{1014}, guarantee that $|f_{r,n}(x)|\leq C_{\alpha_r} C \|(g,G)\|_{SD}$ 
which, in turn, using the continuity of $\Psi$, provides the uniform upper bound
\begin{equation}\label{1025}
\big|\Psi\big(x,f_{r,n}(x)\big)\big|\leq \max \big\{\Psi(x,A): x\in\overline \Omega_r, |A|\leq C_{\alpha_r} C\|(g,G)\|_{SD}\big\},
\end{equation}
for every $x\in\Omega_r$. Here, $C_{\alpha_r}$ is a constant depending only on $\alpha_r$, whereas $C$ is the constant in \eqref{appEST}.
By \eqref{pcconv} and the continuity of $\Psi$, $\Psi(x, f_{r,n}(x))\to \Psi(x, f_r(x))$ in $\Omega_r$ as $n\to\infty$ for every $x$, and \eqref{M11} follows by Lebesgue's theorem on dominated convergence.
This concludes the proof.
%
\end{proof}
\subsection{Spatial localization of the upscaled non-local energy}\label{rto0}
We now turn to the study of
the limit \eqref{1013}, that is, we find an explicit formula for the energy $I^{\alpha_r}(g,G;\Omega_r)$ in the limit as the measure of non-locality $r$ tends to zero.
As mentioned in Sections~\ref{section:preliminaries} and~\ref{section:mathform}, we restrict our attention to continuous functions $\Psi\colon\Omega\times\R{d\times N}\to[0,+\infty)$ belonging to the classes (E) or (L). 
As a first step, given $(g,G)\in SD(\Omega;\R{d}\times\R{d\times N})$, we provide a pair $(\bar g,\bar G)\in BV(\R{N}
;\R{d})\times L^1(\R{N};
\R{d\times N})$ satisfying
\begin{itemize}
\item[(e1)] $(\bar g,\bar G)|_\Omega=(g,G)$;
\item[(e2)] $|D\bar g|(\R{N}
)\leq C\|g\|_{BV(\Omega;\R{d})}$, for some constant $C>0$;
\item[(e3)] $|D^s \bar g|(\partial\Omega)=0$.
\end{itemize}
Because $\partial\Omega$ is Lipschitz and, in particular, $g\in BV(\Omega;\R{d})$, a function $\bar g\in BV(\R{N};\R{d})$  satisfying $\bar g|_\Omega=g$, (e2), and (e3) is provided by \cite[Theorem~1.4]{Gerhardt}. 
Any function $\bar G\in L^1(\R{N}
;\R{d\times N})$ satisfying $\bar G|_\Omega=G$ provides the second element of the pair $(\bar g,\bar G)$ satisfying (e1-3).
For any $(\bar g,\bar G)$ satisfying (e1-3) and for $\alpha_r$ as in \eqref{1014}, in analogy to \eqref{1204}, we define 
\begin{equation}\label{mubar}
\bar\mu\coloneqq (\nabla \bar g-\bar G)\cL^N+D^s\bar g\qquad\text{and}\qquad \bar\mu_r\coloneqq (\bar\mu*\alpha_r)\cL^N,
\end{equation}
where  we observe that the latter expression is well defined for every $x \in \R{N}$.
Moreover, $|\bar\mu|=|\nabla \bar g-\bar G|\cL^N+|D^s\bar g|$.
Finally,
we define the functional $\overline I^{\alpha_r}(\bar g,\bar G;\Omega)$ by
\begin{equation}\label{Ibar}
\begin{split}
\overline I^{\alpha_r}(\bar g,\bar G;\Omega) \coloneqq & \int_{\Omega} \Psi(x,(\bar\mu*\alpha_r)(x))\,\de x.
\end{split}
\end{equation}
Noting that, by (e1), $\bar \mu\res\Omega=\mu$ (see \eqref{1204} and \eqref{mubar}), and recalling \eqref{864}, \eqref{Ibar} can be written as
\begin{equation}\label{Ialpharnew}
\overline I^{\alpha_r}(\bar g,\bar G;\Omega)= I^{\alpha_r}(g,G;\Omega_r)+ \int_{\Omega\setminus \Omega_r} \Psi(x,(\bar\mu*\alpha_r)(x))\,\de x.
\end{equation}
We are now ready to prove Theorem~\ref{final1E}
\begin{proof}[Proof of Theorem~\ref{final1E}]
Let us fix $(g,G)\in SD(\Omega;\R{d}\times\R{d\times N})$ and let $(\bar g,\bar G)\in BV(\R{N};
\R{d})\times L^1(\R{N};
\R{d\times N})$ satisfy (e1-3).
In view of 
\eqref{Ialpharnew}, it suffices to show that $\lim_{r\to0^+} \overline I^{\alpha_r}(\bar g,\bar G;\Omega)$ exists and is equal to the expression for $I(g,G)$ in \eqref{1146} and to show that the integral in the right-hand side of \eqref{Ialpharnew} tends to zero as $r\to0^+$, namely
\begin{equation}\label{estimate}
\lim_{r\to0^+}  \int_{\Omega\setminus \Omega_r} \Psi(x,(\bar\mu*\alpha_r)(x))\,\de x =0.
\end{equation}
Because $\lim_{r\to0^+} \cL^N(\Omega\setminus\Omega_r)=0$, \eqref{estimate} follows by using Fubini's Theorem. 
Indeed, \eqref{1014}, \eqref{convalpha}, \eqref{lT816}, and \eqref{S100} give the following chain of inequalities 
\begin{equation}\label{chain}
\begin{split}
\bigg|  \int_{\Omega\setminus \Omega_r} \Psi(x,(\bar\mu*\alpha_r)(x))\,\de x \bigg| \leq & C_\Psi \int_{\Omega\setminus\Omega_r} \big(1+|(\bar\mu*\alpha_r)(x)|\big)\,\de x \\
\leq & C_\Psi \bigg(\cL^N(\Omega\setminus\Omega_r) +\int_{\Omega\setminus\Omega_r}  |(\bar\mu*\alpha_r)(x)|\,\de x\bigg) \\
\leq & C_\Psi \bigg( \cL^N(\Omega\setminus\Omega_r) + \int_{\Omega\setminus\Omega_r}  \int_{B_r(x)} \alpha_r(y-x)\,\de|\bar\mu|(y)\de x\bigg) \\
\leq & C_\Psi \bigg( \cL^N(\Omega\setminus\Omega_r) + \int_{(\Omega\setminus\Omega_r)^r} \bigg(\int_{B_r(y)}\alpha_r(y-x)\,\de x\bigg) \de|\bar \mu|(y) \bigg) \\ \leq & C_\Psi \Big( \cL^N(\Omega\setminus\Omega_r) + |\bar\mu|\big((\Omega\setminus\Omega_r)^r\big)
\Big),
\end{split}
\end{equation}
where we recall that, given a set $A\subseteq\R{N}$, we define $A^r\coloneqq A+B_r$.
The last term above converges to $0$ as $r \to 0^+$, since $(\Omega\setminus\Omega_r)^r\to\partial\Omega$ as $r\to0^+$ and $\cL^N(\partial \Omega)=0$, and $|\bar{\mu}|(\partial \Omega)=0$ by the lower semicontinuity of the total variation with respect to the weak-* convergence of measures (see \eqref{mubar}) and (e3).
We now prove that
\begin{equation}\label{T555}
\lim_{r\to0^+} \overline I^{\alpha_r}(\bar g,\bar G;\Omega) = \int_\Omega \Psi\big(x,\nabla g(x) - G(x)\big)\,\de x+\int_{\Omega\cap S_g}\Psi^{\infty}\Big(x,\frac{\de D^s g}{\de|D^s g|}(x)\Big)\,\de|D^s g|(x).
\end{equation}
To do so, let us define the functional $\overline \sI\colon\cM(\overline \Omega
;\R{d\times N})\to[0,+\infty)$ by
\begin{equation}\label{1213}
\overline\sI(\lambda)\coloneqq \int_\Omega \Psi\Big(x,\frac{\de\lambda}{\de\cL^N}(x)\Big)\,\de x+ \int_{\overline\Omega \cap\spt (|\lambda^s|)} \Psi^\infty\Big(x,\frac{\de\lambda}{\de|\lambda^s|}(x)\Big)\,\de|\lambda^s|(x),
\end{equation}
for $\lambda\in\cM(\overline \Omega
;\R{d\times N})$, where 
$\Psi^\infty$ denotes the recession function at infinity of $\Psi$ (see \eqref{defrecession2}).
Keeping \eqref{mubar}, \eqref{Ibar}, and \eqref{1213} in mind, we can write $\overline I^{\alpha_r}(\bar g,\bar G;\Omega)=\overline\sI(\bar\mu_r)$; similarly, invoking (e1), the right-hand side of \eqref{T555} can be written as $\overline\sI(\bar\mu)$, so that \eqref{T555} is proved if we show that 
\begin{equation}\label{T556}
\lim_{r\to0^+} \overline\sI(\bar\mu_r)=\overline\sI(\bar\mu).
\end{equation}
Recalling the definitions of $\bar\mu,\bar\mu_r$ in \eqref{mubar}, we 
use 
\eqref{lT816} to obtain the estimate
$$
|\bar \mu_r|(\overline{V}) 
\leq \|\nabla \bar g- \bar G\|_{L^1(V^r;\R{d})}+|D^s \bar g|(V^r),
$$
for every open set $V \subseteq \mathbb R^N$.
In turn, 
Theorem~\ref{815}(i) and Proposition \ref{AFPb}(b) 
entail that $\bar \mu_r$ converges locally weakly-*  in $\cM(\overline\Omega;\R{d\times N})$ to $\bar\mu$ 
and $\bar \mu_r(\bar\Omega)\to \bar \mu(\bar \Omega )$.
Moreover, $|\bar\mu_r|$ converges locally weakly-* in $\mathcal M^+(\mathbb R^N)$ to $|\bar\mu|=|\nabla \bar g-\bar G|\cL^N+|D^s\bar g|$.
Furthermore, since $|\bar \mu|(\partial \Omega)\leq \|\bar g\|_{BV(\partial \Omega;\mathbb R^d)}+ \|\bar G\|_{L^1(\partial\Omega;\R{d\times N})}=0,$
\[\begin{split}
|\bar \mu|(\Omega)=|\bar \mu|(\overline\Omega)\leq \liminf_{r \to +\infty}|\bar \mu_r|(\overline\Omega) 
\leq &\; \liminf_{r \to +\infty}\big(|D^s \bar g|(\Omega^r)+\|\bar \nabla \bar g- \bar G\|_{L^1(\Omega^r;\R{d\times N})}\big) \\
=&\; |D^sg|(\Omega)+\| \nabla g- G\|_{L^1(\Omega;\R{d\times N})}.
\end{split}\]
This, together with Proposition~\ref{AFPb}(a) gives 
\begin{equation}\label{427bis}
|\bar\mu_r| \wsto |\nabla \bar g-\bar G|\mathcal L^N+ |D^s \bar g|\qquad \hbox{ in }\mathcal M^+(\overline\Omega).
\end{equation} 
Finally, 
by Theorem~\ref{815}(ii) we obtain that $\langle \bar\mu_r\rangle(\overline\Omega)\to\langle\bar\mu\rangle(\overline\Omega)$, yielding that 
$\bar\mu_r$ 
$\langle\cdot\rangle$-strict converges to $\bar\mu$ 
(see Definition~\ref{1100}(iv)). 

If $\Psi$ belongs to the class (E), since the $\limsup$ in the definition of $\Psi^\infty$ is indeed a limit (see Remark~\ref{tuttiuguali}) we can apply Theorem~\ref{ReshetnyaktheoremKR}, to obtain \eqref{T556}.
In turn \eqref{T555} is proved and therefore \eqref{1146}, which concludes the proof.

If $\Psi$ belongs to the class (L), Theorem~\ref{ReshetnyaktheoremBCMS} provides the upper bound
\begin{equation}\label{1464}
\begin{split}
\limsup_{r \to 0^+}  \int_{\Omega} \Psi\big(x,(\bar\mu*\alpha_r)(x)\big)\,\de x \leq & \int_\Omega \Psi\big(x,(\nabla \bar g - \bar G)(x)\big)\,\de x+\int_{\Omega\cap S_{\bar g}}\Psi^{\infty}\Big(x,\frac{\de D^s \bar g}{\de|D^s \bar g|}(x)\Big)\,\de|D^s \bar g|(x) \\
= & \int_\Omega \Psi\big(x,(\nabla  g -  G)(x)\big)\,\de x+\int_{\Omega\cap S_{ g}}\Psi^{\infty}\Big(x,\frac{\de D^s  g}{\de|D^s  g|}(x)\Big)\,\de|D^s  g|(x),
\end{split}
\end{equation}
where the equality holds by (e1).
We now prove that
\begin{equation}\label{1470}
\int_\Omega \Psi\big(x,\nabla  g(x) -  G(x)\big)\,\de x+\int_{\Omega\cap S_{ g}}\Psi^{\infty}\Big(x,\frac{\de D^s  g}{\de|D^s  g|}(x)\Big)\,\de|D^s  g|(x) \leq \liminf_{r \to 0^+}  \int_{\Omega} \Psi\big(x,(\bar\mu*\alpha_r)(x)\big)\,\de x.
\end{equation}
To this end, consider the measures $\bar \theta_r \in \cM^+(\Omega)$ defined by $\bar \theta_r \coloneqq \Psi(\cdot,(\bar\mu*\alpha_r)(\cdot))\cL^N$. Since they form a bounded family of Radon measures, they converge weakly-* to some positive measure $\bar \theta$.
We obtain \eqref{1470} if we show that
\begin{subequations}\label{808}
\begin{eqnarray}
\frac{\de \bar\theta}{\de\cL^N} (x) & \!\!\!\! \geq &  \!\!\!\!\Psi\big(x,(\nabla g - G)(x)\big) \quad\text{for $\cL^N$-a.e.~$x \in \Omega$,} \label{lowerbulk}\\
\frac{\de \bar\theta}{\de |D^s g|} (x) & \!\!\!\! \geq & \!\!\!\! \Psi^{\infty}\Big(x,\frac{\de D^s g}{\de|D^s g|}(x)\Big) \quad \text{for $\cH^{N-1}$-a.e.~$x \in S_g$.} \label{lowerinterf}
\end{eqnarray}
\end{subequations}
We start with \eqref{lowerbulk}.
By the linearity of the convolution operator and the definition of $\bar\mu_r$, we know that, as $r\to0^+$,
\begin{equation}\label{S110}
\big((\nabla \bar g-\bar G)\cL^N * \alpha_r\big)\cL^N\wsto (\nabla g-G)\cL^N \qquad\text{and}\qquad (D^s \bar g*\alpha_r)\cL^N\wsto D^s g,
\end{equation}
in $\cM(\Omega;\R{d\times N})$ and, by \cite[Corollary~2.1.17]{Grafakos}, we have
\begin{equation}\label{809}
(\nabla \bar g-\bar G)\cL^N * \alpha_r
\to (\nabla g-G)(x)\qquad \text{for $\cL^N$-a.e.~$x\in\Omega$}.
\end{equation}
Let us fix $x_0\in\Omega\setminus S_g$ which is a Lebesgue point for $\nabla g-G$ and let us compute
\begin{equation*}
\begin{split}
\frac{\de \bar \theta}{\de\cL^N}&(x_0) = \lim_{k\to \infty}\frac{\bar\theta(Q(x_0;\delta_k))}{\cL^N(Q(x_0;\delta_k))} = \lim_{k\to\infty}\lim_{r\to0^+} \frac{\bar\theta_r(Q(x_0;\delta_k))}{\delta_k^N}=  \lim_{k\to\infty}\lim_{r\to0^+} \frac{1}{\delta_k^N}\int_{Q(x_0;\delta_k)}  \Psi(x,(\bar \mu*\alpha_r)(x))\,\de x \\
=&   \lim_{k\to\infty}\lim_{r\to0^+} \frac{1}{\delta_k^N}\int_{Q(x_0;\delta_k)}  \Psi\big(x,((\nabla \bar g-\bar G)\cL^N * \alpha_r)(x)+ (D^s \bar g*\alpha_r)(x) \big)\,\de x \\
\geq &  \lim_{k\to\infty}\lim_{r\to0^+} \frac{1}{\delta_k^N}\int_{Q(x_0;\delta_k)}  \Psi\big(x,((\nabla \bar g-\bar G)\cL^N * \alpha_r)(x) \big)\,\de x  - \lim_{k\to\infty}\lim_{r\to0^+} \frac{L_\Psi}{\delta_k^N}\int_{Q(x_0;\delta_k)} \!\! |(D^s \bar g*\alpha_r)(x)|\,\de x,
\end{split}
\end{equation*}
where we have used \eqref{S102}.
Since, by the second convergence in \eqref{S110}, the last integral is the Radon-Nikod\'ym derivative of $|D^s \bar g|$ with respect to $\cL^N$, it vanishes, so that we have
\begin{equation*}
\begin{split}
\frac{\de \bar \theta}{\de\cL^N}(x_0) \geq & \lim_{k\to\infty}\lim_{r\to0^+} \frac{1}{\delta_k^N}\int_{Q(x_0;\delta_k)}  \Psi\big(x,((\nabla \bar g-\bar G)\cL^N * \alpha_r)(x) \big)\,\de x \\
\geq & \lim_{k\to\infty} \frac{1}{\delta_k^N}\int_{Q(x_0;\delta_k)}  \Psi\big(x,(\nabla g-G)(x) \big)\,\de x \\
\geq &  \lim_{k\to\infty}\int_{Q} \Psi\big(x_0+\delta_k y,(\nabla g-G)(x_0+\delta_k y) \big)\,\de y \geq \Psi\big(x_0,(\nabla g-G)(x_0) \big),
\end{split}
\end{equation*}
where we have used the continuity of $\Psi$, \eqref{809}, and Fatou's Lemma in the second inequality, a change of variables and \eqref{S102} and \eqref{S103} in the subsequent estimates.
This proves \eqref{lowerbulk}.
To prove \eqref{lowerinterf}, let us fix $x_0 \in S_g$ and let $\tau(x_0) \coloneqq \displaystyle\frac{ \de D^s g}{\de|D^s g|}(x_0)$.
By Lemma~\ref{S105}, the recession function $\Psi^\infty(x_0,\tau(x_0))$ can be computed using formula \eqref{defrecession3}.
Let now $t_k\in\R{}$ be a sequence diverging to $+\infty$ as $k\to\infty$ along which the $\limsup$ in \eqref{defrecession3} is indeed a limit, that is,
\begin{equation*}
\Psi^\infty(x_0,\tau(x_0)) = \lim_{k \to \infty} \frac{\Psi(x_0,t_k\tau(x_0))}{t_k}.
\end{equation*} 
Since $\Psi$ is Lipschitz continuous, a reasoning analogous to that of \cite[Lemma~4.2]{BCMS2013} grants that the $t_k$'s can be chosen as
\begin{equation}\label{1493}
t_k \coloneqq \frac{|D^s \bar g|(Q(x_0;\delta_k))}{\delta_k^N},
\end{equation}
with $\delta_k\in\R{}$ a vanishing sequence such that $\bar \theta(\partial Q(x_0;\delta_k))=0$.
Then
\begin{equation*}
\begin{split}
\frac{\de \bar \theta}{\de |D^s g|}(x_0) & = \lim_{k\to \infty}\frac{\bar\theta(Q(x_0;\delta_k))}{|D^s g|(Q(x_0;\delta_k))} = \lim_{k\to\infty}\lim_{r\to0^+} \frac{\bar \theta_r(Q(x_0;\delta_k))}{|D^s g|(Q(x_0;\delta_k))} \\
&= \lim_{k\to\infty}\lim_{r\to0^+} \frac{\displaystyle\int_{Q(x_0;\delta_k)} \Psi\big(x,(\bar\mu*\alpha_r)(x)\big)\,\de x}{|D^s g|(Q(x_0;\delta_k))} \\ 
&= \lim_{k\to\infty}\lim_{r\to0^+}  \frac{1}{t_k} \int_Q \Psi \big(x_0+\delta_k y,(\bar\mu*\alpha_r)(x_0 + \delta_k y)\big)\,\de y
\end{split}
\end{equation*}
where the last equality follows by a change of variables, taking \eqref{1493} into account.
Defining
$$\bar w_{k,r}(y)\coloneqq \frac{(\bar\mu*\alpha_r)(x_0 + \delta_k y)}{t_k},$$
we can continue the chain of equalities above as follows
\begin{equation*}
\begin{split}
\frac{\de \bar\theta}{\de |D^s g|}(x_0)  =  \lim_{k\to\infty}\lim_{r\to0^+} & \frac{1}{t_k} \int_Q \Psi \big(x_0+\delta_k y,(\bar\mu*\alpha_r)(x_0 + \delta_k y)\big)\,\de y \\
= \lim_{k\to\infty}\lim_{r\to0^+} & \frac{1}{t_k} \int_Q \Psi \big(x_0+\delta_k y,t_k \bar w_{k,r}(y)\big)\, \de y\\
\geq  \lim_{k\to\infty}\lim_{r\to0^+} & \bigg[\frac{1}{t_k} \int_Q \Psi \big(x_0,t_k\tau(x_0)\big)\, \de y - L_\Psi \int_Q \big|\bar w_{k,r}(y) - \tau(x_0) \big| \, \de y \\
& -\frac1{t_k}\int_Q \omega(\delta_k|y|)(1+t_k|\tau(x_0)|)\,\de y\bigg] = \Psi^\infty(x_0,\tau(x_0)),
\end{split}
\end{equation*}
where we have used \eqref{S102} and \eqref{S103} and where the last two terms in the square bracket vanish since $\lim_{k\to\infty}\lim_{r\to0^+}\int_Q w_{k,r}(y)\, \de y = \tau(x_0)$ by \eqref{1493} and by the properties of the modulus of continuity $\omega$.
This concludes the proof of \eqref{lowerinterf} and, consequently, of \eqref{1470}.
Combining \eqref{1464} and \eqref{1470} yields a chain of equalities, which is precisely \eqref{1146}. 
The theorem is proved.
\end{proof}
Recalling \eqref{810}, 
the limiting energy $I(g,G)$ in \eqref{1146} can be written as
\begin{equation}\label{811}
\begin{split}
I(g,G)= 
\int_\Omega \Psi\big(x,\nabla g(x) - G(x)\big)\,\de x+\int_{\Omega\cap S_g}\Psi^{\infty}\big(x,[g](x)\otimes\nu_g(x)\big)\,\de\cH^{N-1}(x).
\end{split}
\end{equation}
Moreover, as a particular case of $\Psi$ with sublinear growth, one can consider a bounded $\Psi$. In this case, the formula above reduces  to \eqref{624} (since $\Psi^\infty=0$). 
\section{Coupling local and non-local energies}\label{together}
In this section we extend the results first proved in the pioneering paper \cite{CF1997} to the case of $x$-dependent energy densities.
The integral representation results \cite[Theorems~ 2.16 and~2.17]{CF1997} are expected to hold with the obvious modifications, namely with the relaxed/upscaled energy densities depending on $x$ as well.
This generalization is somewhat natural and can be obtained with minor modifications to the original proofs, but since it is not presented elsewhere, we highlight here the adaptation of the proofs from \cite{CF1997} for sake of completeness.
\subsection{Relaxation/upscaling of the local energy $E_L$}\label{ChoksiFonseca}
In this subsection we present the relaxation/upscaling results for local energies, like $E_L$ defined in \eqref{1002}, contained in the paper \cite{CF1997}.
We start by introducing the assumptions on the bulk and interfacial energy densities $W$ and $\psi$.
Let $p\geq1$ and let $W\colon \Omega\times\R{d\times N} \to [0, +\infty[$ and $\psi\colon \Omega\times\R{d}\times \S{N-1} \to [0, +\infty[$ be continuous functions satisfying the following conditions 
\begin{itemize}
\item[$(W1)_p$ ]there exists $C >0$ such that, for all $x\in\Omega$ and $A, B \in \R{d\times N}$,
\begin{equation*}
|W(x,A) - W(x,B)| \leq C| A - B| \big(1+|A|^{p-1}+|B|^{p-1}\big)
\end{equation*}
\item[$(W2)$] there exists a continuous function $\omega_W\colon[0,+\infty)\to[0,+\infty)$ with $\omega_W(s)\to 0$ as $s\to0^+$ such that, for every $x,x_0\in\Omega$ and $A\in\R{d\times N}$,
\begin{equation*}
|W(x,A)-W(x_0,A)|\leq \omega_W(|x-x_0|)(1+|A|^p);
\end{equation*}
\item[$(W3)$] there exist $C,T>0$ and $0 < \alpha < 1$ such that, for all $x\in\Omega$ and $A\in\R{d\times N}$ with $|A|=1$, and, if $p=1$,
\begin{equation*}
\bigg |W^{\infty}(x,A) - \frac{W(x,tA)}{t}\bigg| \leq \frac{C}{t^{\alpha}},\qquad\text{for all $t>T$,}
\end{equation*}
where $W^{\infty}$ denotes the \textit{recession function} at infinity of $W$ (with respect to $A$), see \eqref{defrecession3};
\item[$(\psi1)$] there exist $c,C > 0$ such that, for all $x\in\Omega$, $\lambda \in \R{d}$, and $\nu \in \S{N-1}$,
$$c|\lambda| \leq \psi(x,\lambda, \nu) \leq C|\lambda |;$$
\item[$(\psi2)$] (positive $1$-homogeneity) for all $x\in\Omega$, $\lambda \in \R{d}$, $\nu \in \S{N-1}$, and $t >0$
$$\psi(x,t\lambda, \nu) = t\psi(x, \lambda, \nu),$$
\item[$(\psi3)$] (sub-additivity) for all $x\in\Omega$, $\lambda_1,\lambda_2 \in \R{d}$, and $\nu \in \S{N-1}$,
\begin{equation*}
\psi(x, \lambda_1 + \lambda_2, \nu) \leq \psi(x,\lambda_1, \nu) +\psi(x,\lambda_2, \nu).
\end{equation*}
\item[$(\psi4)$] there exists a continuous function $\omega_\psi\colon[0,+\infty)\to[0,+\infty)$ with $\omega_\psi(s)\to 0$ as $s\to0^+$ such that, for every $x_0\in\Omega$, $\lambda \in \R{d}$, and $\nu \in \S{N-1}$,
\begin{equation*}
|\psi(x,\lambda,\nu)-\psi(x_0,\lambda,\nu)|\leq\omega_\psi(|x-x_0|)|\lambda|.
\end{equation*}
\end{itemize}
Given $W$ and $\psi$ as above, and $u\in SBV(\Omega;\R{d})$, we defined the initial energy $E_L(u)$ as
\begin{equation}\label{901}
E_L(u)\coloneqq \int_\Omega W(x,\nabla u(x))\,\de x + \int_{\Omega\cap S_u} \psi(x,[u](x), \nu_u(x))\,\de\mathcal{H}^{N-1}(x)
\end{equation}
and, given $(g,G)\in SD(\Omega;\R{d}\times\R{d\times N})$, we defined the relaxed energies $I_p(g,G)$ 
as
\begin{equation}\label{910}
\begin{split}
I_p(g,G)\coloneqq \inf_{\{u_n\}\subset SBV(\Omega;\R{d})}\Big\{ \liminf_{n\to\infty} E_L(u_n): &\; \text{$u_n\to(g,G)$ in the sense of \eqref{appCF},} \\
&\; (1-\delta_1(p)) \sup_n \|\nabla u_n\|_{L^p(\Omega;\R{d\times N})}
<\infty\Big\}.
\end{split}
\end{equation}
In the formula above, and in what follows, we use the symbol $\delta_1(p)$ as the Kronecker delta computed at $p$, namely $\delta_1(p)=1$ if $p=1$ and zero otherwise, and use it as a selector between the cases $p=1$ and $p>1$.
In particular, in \eqref{910}, the control on the $L^p$ norm of $|\nabla u_n|$ does not appear in the formula if $p=1$, since in that case $1-\delta_1(p)=0$.
We introduce now the classes of competitors for the cell formulae for the relaxed/upscaled bulk and surface energy densities.
For $A,B\in\R{d\times N}$ let
\begin{equation}\label{T001}
\cC_p^{\bulk}(A,B)\coloneqq \bigg\{u\in SBV(Q;\R{d}): u|_{\partial Q}(x)=Ax, \int_Q \nabla u\,\de x=B, |\nabla u|\in L^p(Q) \bigg\} 
\end{equation}
and for $\lambda\in\R{d}$ and $\nu\in\S{N-1}$ let
\begin{equation*}
\cC_p^\surface(\lambda,\nu)\coloneqq \Big\{u\in SBV(Q_\nu;\R{d}): u|_{\partial Q_\nu}(x)=u_{\lambda,\nu}(x), \delta_1(p)\fC_1(u)+(1-\delta_1(p))\fC(u)\Big\},
\end{equation*}
where the function $u_{\lambda,\nu}$ is defined by
\begin{equation*}
u_{\lambda,\nu}(x)\coloneqq 
\begin{cases}
\lambda & \text{if $x\cdot\nu\geq0$,} \\
0 & \text{if $x\cdot\nu<0$,}
\end{cases}
\end{equation*}
and the conditions $\fC_1(u)$ and $\fC(u)$ are
\begin{equation}\label{T003}
\fC_1(u) \Longleftrightarrow \int_Q \nabla u\,\de x=0\qquad\text{and}\qquad \fC(u) \Longleftrightarrow \nabla u(x)=0\;\text{for $\cL^N$-a.e.~$x\in Q_\nu$}.
\end{equation}
We state now the integral representation theorem for the relaxed/upscaled energies $I_p$ defined in \eqref{910}.
It generalizes the results contained in \cite[Theorems~2.16 and~2.17]{CF1997} to the inhomogeneous case considered here.
For the sake of being concise, we give a unified statement through the use of the selector $\delta_1(p)$, which takes into account the different nuances between the case $p=1$ and the case $p>1$. 
Note that the formulae for the relaxed energy densities $H_p$ and $h_p$ are obtained via the blow-up method \cite{BBBF1996,FM1992,FM1993} and involve the contributions of both $W$ and $\psi$ for $H_p$, and of $\psi$ and possibly $W^\infty$ for $h_p$.
\begin{theorem}\label{903}
Let $p\geq1$ and let $W\colon \Omega\times\R{d\times N} \to [0, +\infty[$ and $\psi\colon \Omega\times\R{d}\times \S{N-1} \to [0, +\infty[$ be continuous functions satisfying hypotheses $(W1)_p$, $(W2)$, $\delta_1(p)(W3)$, $(\psi1)$, $(\psi2)$, $(\psi3)$, and $(\psi4)$; let $(g,G)\in SD(\Omega;\R{d}\times\R{d\times N})$ with $G\in L^p(\Omega;\R{d\times N})$ and let $I_p(g,G)$ be given by \eqref{910}.
Then there exist $H_p\colon\Omega\times \R{d\times N}\times \R{d\times N}\to[0,+\infty)$ and $h_p\colon\Omega\times\R{d}\times\S{N-1}\to[0,+\infty)$ such that 
\begin{equation}\label{904}
I_p(g,G)=\int_\Omega H_p(x,\nabla g(x),G(x))\,\de x+\int_{\Omega\cap S_g} h_p(x,[g](x),\nu_g(x))\,\de\cH^{N-1}(x).
\end{equation}
For all $x_0\in\Omega$ and $A,B\in\R{d\times N}$,
\begin{equation}\label{906}
H_p(x_0,A,B)\coloneqq \inf\bigg\{ 
\int_Q W(x_0,\nabla u(x))\,\de x+\int_{Q\cap S_u} \psi(x_0,[u](x),\nu_u(x))\,\de\cH^{N-1}(x): 
u\in\cC_p^\bulk(A,B)\bigg\};
\end{equation}
for all $x_0\in\Omega$, $\lambda\in\R{d}$, and $\nu\in\S{N-1}$,
\begin{equation}\label{907}
\!\!\! h_p(x_0,\lambda,\nu)\coloneqq \inf\bigg\{ 
\delta_1(p) \!\! \int_{Q_\nu} \!\!\!\! W^\infty(x_0,\nabla u(x))\,\de x+ \! \int_{Q_\nu\cap S_u} \!\!\!\!\!\!\!\!\! \psi(x_0,[u](x),\nu_u(x))\,\de\cH^{N-1}(x): 
u\in\cC_p^\surface(\lambda,\nu)\bigg\},
\end{equation}
with $W^\infty$ defined in \eqref{defrecession3}. 
\end{theorem}
\begin{remark}\label{rmkCF}
Theorem~\ref{903} collects the content of Theorems~2.16 and 2.17 
in \cite{CF1997} in a compact form.
In particular, the form of the integral representation of the relaxed/upscaled energies 
\eqref{910} provided by formula \eqref{904} is structurally the same both for $p=1$ and for $p>1$: it features a bulk energy and an interfacial energy.
We make the following observations.
\begin{itemize}
\item The condition $|\nabla u|\in L^p(Q)$ in \eqref{T001} is redundant if $p=1$ (see \cite[Remark~2.15]{CF1997});
\end{itemize}
\begin{itemize}
\item If $p=1$, hypothesis $(W3)$ is required and 
we notice that in formula \eqref{907} the recession function at infinity $W^\infty$ defined in \eqref{defrecession3} appears, to account for concentration phenomena arising when taking the limit of functions in $L^1$.
\item In \eqref{T003}, condition $\fC_1$ contains condition $\fC$, so that, for every $\lambda\in\R{d}$ and $\nu\in\S{N-1}$, we have the inclusion $\cC_{p>1}^\surface(\lambda,\nu)\subset\cC_1^\surface(\lambda,\nu)$. 
\item The cell formula \eqref{907} for $p>1$ corrects formula (2.17) in \cite{CF1997}, where the dependence on the normal $\nu$ was mistakenly omitted, as already noted in \cite[Theorem~3]{OP2015} and \cite[formula~(4)]{S17}.
\end{itemize}
We point out the following final remarks.
\begin{itemize}
\item Hypothesis $(W1)_p$ could be strengthened to include coercivity ($p$-growth from below). Although this would be a strong restriction from the mechanical point of view, it would make the proofs easier. We refer the reader to \cite[Step~1 in the proof of Proposition~2.22]{CF1997} for a discussion on this.
\item If $p>1$, hypotheses $(\psi1)$ and $(\psi2)$ 
can be relaxed. We refer the reader to \cite[Remark~3.3]{CF1997} for a discussion on this. 
\item If $p>1$, Theorem~\ref{903} provides a representation of the relaxed/upscaled energy density $I_p(g,G)$ only in the case $G\in L^p(\Omega;\R{d\times N})$ (see again \cite[Remark~2.15]{CF1997}).
\end{itemize}
These final remarks pave the way for a statement of Theorem~\ref{903} under the minimal set of hypotheses.
\qed
\end{remark}
\begin{proof}[Sketch of the proof of Theorem~\ref{903}]
Formula \eqref{904} is obtained by using the blow-up method \cite{BBBF1996,FM1992,FM1993} to prove that the energy densities \eqref{906} and \eqref{907} provide upper and lower bound for the Radon-Nikod\'ym derivatives of suitable measures associated with $I_p(g,G)$ with respect to $\cL^N$ and $|[g]|\cH^{N-1}\res S_g$.
The dependence on $x$ is not involved in this process, and the existence of the moduli of continuity $\omega_W$ and $\omega_\psi$ is a strong enough assumption to estimate the error when passing from the evaluation of the energy densities at generic $x\in Q(x_0,\delta)$ to the evaluation at $x_0$. A similar strategy was undertaken in \cite{BMMO2017}, in the spirit of \cite{BBBF1996}.
\end{proof}
We remark that Theorem~\ref{903}  does not address effects such as the bending due to jumps in $\nabla u_n$, that are captured by second-order structured deformations \cite{BMMO2017,FHP,OP2000}.
\subsection{Relaxation/upscaling of the total energy $E_L+E^{\alpha_r}$}
We now address the relaxation/upscaling of the total energy including both the local initial energy $E_L$ and the non-local initial energy $E^{\alpha_r}$. 
By Remark~\ref{coupling} below we can perform the relaxation/upscaling of $E_L$ and the upscaling of $E^{\alpha_r}$ as two separate processes.
\begin{remark}\label{coupling}
We note that given  $(g,G)\in SD(\Omega;\R{d}\times\R{d\times N})$, any sequence of functions $u_n\in SBV(\Omega;\R{d})$ admissible for the relaxation/upscaling process of Theorem~\ref{903}, such that $\liminf_{n\to\infty} E_L(u_n)$ is finite,  belongs to $\Ad(g,G)$. 
In fact, by \eqref{910}, $u_n$ converges to $(g,G)$ in the sense of \eqref{appCF}, thus providing a uniform bound on the $L^1$ norm of $u_n$ and on the $L^p$ norm of $\nabla u_n$. Then, the coercivity of $\psi$ in $(\psi1)$ ensures that $|D^s u_n|$ is also a unifomly bounded sequence of measures, hence $D u_n$ 
is bounded in total variation and therefore has a weakly-* converging subsequence (not relabelled) such that $D^s u_n \wsto (\nabla g - G)\cL^N + D^s g$, that is, \eqref{centerline} holds true. Moreover, by the metrisability on compact sets of the weak-* convergence, (see \cite[Remark~1.57, Theorem~1.59, and subsequent comments]{AFP}) and Urysohn's principle, the whole sequence $u_n$ belongs to $\Ad(g,G)$.
Moreover, the collection of such sequences $u_n$ is non-empty, ad the infimum over such $u_n$ of $\liminf_{n\to\infty} E_L(u_n)$ equals $I_L(g,G)$.
\qed
\end{remark}
\begin{theorem}\label{main2}
Under the conditions of Theorem~\ref{mainrho} and Theorem~\ref{903}, the relaxation/upscaling \eqref{1037} of the initial energy \eqref{1036} admits the integral representation \eqref{1038}, where, for any $(g,G)\in SD(\Omega;\R{d}\times\R{d\times N})$, the relaxed/upscaled energy $I_L(g,G)$ of the local initial energy $E_L$ in \eqref{1002} is given by \eqref{904} and the upscaled energy $I^{\alpha_r}(g,G;\Omega_r)$ of the non-local initial energy $E^{\alpha_r}$ is provided by Theorem~\ref{mainrho}.
In particular, 
\begin{equation}\label{623}
\begin{split}
J^{\alpha_r}(g,G) = & \int_\Omega H_p(x,\nabla g(x),G(x))\,\de x + \int_{\Omega\cap S_g} h_p(x,[g](x),\nu_g(x))\,\de\cH^{N-1}(x) \\
& +\int_{\Omega_r} \Psi\big(x,((\nabla g-G)\ast \alpha_r)(x)+(D^s g\ast\alpha_r)(x)\big) \,\de x.
\end{split}
\end{equation}
\end{theorem}
\begin{proof}
The representation formula \eqref{623} is an immediate consequence of Theorem~\ref{mainrho}, Theorem~\ref{903}, Remark ~\ref{coupling}, and the superadditivity properties of the $\liminf$.
\end{proof}
\begin{corollary}\label{finalfinal}
Under the conditions of Theorem~\ref{final1E} and Theorem~\ref{main2}, for any $(g,G)\in SD(\Omega;\R{d}\times\R{d\times N})$, the functional $J(g,G)$ defined in \eqref{1039} admits the integral representation in \eqref{1040}, namely, 
\begin{equation}\label{812}
\begin{split}
J(g,G) = &  \int_\Omega H_p(x,\nabla g(x),G(x))\,\de x + \int_{\Omega\cap S_g} h_p(x,[g](x),\nu_g(x))\,\de\cH^{N-1}(x) \\
& + \int_\Omega \Psi\big(x,\nabla g(x) - G(x)\big)\,\de x+
\int_{\Omega\cap S_g} \Psi^{\infty}\big(x,[g](x)\otimes \nu_g(x) \big)\,\de\cH^{N-1}(x). 
\end{split}
\end{equation}
\end{corollary}
\begin{proof}
The result follows immediately by Theorem~\ref{final1E} and Theorem~\ref{main2}. 
\end{proof}
As noticed in the last bullet of Remark~\ref{rmkCF}, if $p>1$ Theorem~\ref{main2} and Corollary~\ref{finalfinal} provide integral representation results only for fields $G\in L^p(\Omega;\R{d\times N})$.
\subsection{On the reverse order of the limits}\label{reversed}
After the presentation of the iterated limiting procedure carried out in Sections~\ref{section:statement of the problem} and~\ref{rto0}, a legitimate question is whether the two operations commute, namely, whether we obtain the same result if we reverse the order in which the two limits are taken: first letting $r\to0^+$ and then letting $n\to\infty$.
The problem is a relevant one in the scientific community 
and a similar question was studied in \cite{CMMO} for a problem of dimension reduction in the context of structured deformations.
In the following few lines, we will give a brief explanation of why in the present case a commutability result does not hold.
Under the hypotheses of the previous sections on $W$, $\psi$, and $\Psi$,
let us consider the reversed iterated limiting procedure for an initial energy of the type $E_L+E^{\alpha_r}$, with $E_L$ as in \eqref{901} and $E^{\alpha_r}$ as in \eqref{1009}.
We first let the measure of non-locality tend to zero and then relax/upscale to structured deformations, namely we consider, for $u\in SBV(\Omega;\R{d})$
\begin{equation}\label{T100}
I_L(u)\coloneqq \lim_{r\to0} \big( E_L(u)+E^{\alpha_r}(u) \big)
\end{equation}
and then we relax/upscale this energy as in \eqref{910}, for $(g,G)\in SD(\Omega;\R{d}\times\R{d\times N})$:
\begin{equation}\label{T101}
\begin{split}
I_p^{(R)}(g,G)\coloneqq \inf_{\{u_n\}\subset SBV(\Omega;\R{d})}\Big\{ \liminf_{n\to\infty} I_L(u_n): &\; \text{$u_n\to(g,G)$ in the sense of \eqref{appCF},} \\
&\; (1-\delta_1(p)) \sup_n \|\nabla u_n\|_{L^p(\Omega;\R{d\times N})} 
<\infty\Big\}.
\end{split}
\end{equation}
Given that $E_L$ is independent of $r$, it is easy to deal with \eqref{T100}.
Since $\Psi$ belongs to the class (E) or (L), an application of the Reshetnyak Continuity Theorem~\ref{ReshetnyaktheoremKR} with $\Phi=\Psi$ gives
\begin{equation*}
\begin{split}
\lim_{r\to0} E^{\alpha_r}(u)= &
\lim_{r\to 0} \int_\Omega \Psi \big(x, (\alpha_r \ast D^s u)(x)\big) \,\de x =\int_\Omega \Psi (x, 0)\,\de x + \int_{\Omega \cap S_{u}} \Psi^\infty \Big( x, \frac{\de D^su}{\de |D^s u|}(x)\Big)\,\de |D^s u|(x) \\
=&  \int_\Omega \Psi (x, 0)\,\de x + \int_{\Omega \cap S_{u}} \Psi^\infty \big( x, [u](x)\otimes\nu_u(x)\big)\,\de \cH^{N-1}(x).
\end{split}
\end{equation*}
The chain of equalities above is justified upon extending the function $u$ outside of $\Omega$, as it was done in Section~\ref{rto0} for the function $g$ (through the application of \cite[Theorem~1.4]{Gerhardt}), and recalling that the energy does not depend on the chosen extension in the limit as $r\to0$ (see also \eqref{chain}).
Therefore, in \eqref{T100} we obtain
\begin{equation*}
I_L(u)=E_L(u)+\int _\Omega \Psi (x, 0)\,\de x + \int_{\Omega \cap S_{u}} \Psi^\infty \big( x, [u](x)\otimes\nu_u(x)\big)\,\de \cH^{N-1}(x).
\end{equation*}
Now, it is easy to prove that if $\Psi$ is in the class (L), then also $\Psi^\infty$ is in the class (L). 
Therefore, for either (i) $\Psi$ belonging to (L) or (ii) $\Psi$ belonging to (E) with $\Psi^\infty$ Lipschitz in the second variable uniformly with respect to the first one, it is immediate to see that $\Psi^\infty$ is a surface energy density that satisfies hypotheses $(\psi1)$, $(\psi2)$, and $(\psi3)$ (see \cite[Remark 3.3]{CF1997} and \cite[Remark 3.1]{MMZ}). 
Thus, the relaxation/upscaling process \eqref{T101} is the same as that of Theorem~\ref{903} for a local energy of the type \eqref{901} whose densities are $\widetilde W(x,A)\coloneqq W(x,A)+\Psi(x,0)$ and $\widetilde\psi(x,\lambda,\nu)\coloneqq\psi(x,\lambda,\nu)+\Psi^\infty(x,\lambda\otimes\nu)$.
The cell formulas \eqref{906} and \eqref{907} imply that only the behavior of $\Psi(x,A)$ at $A=0$ or as $|A|\to\infty$ can influence the relaxed/upscaled energy in \eqref{T101}, whereas the presence of the third integral in \eqref{812} shows that all of the values of $\Psi(x,A)$ can influence the relaxed/upscaled energy $J(g,G)$ in \eqref{812}.
\subsection{Bulk relaxed densities of the form 
$F_1(x,G(x))+F_2(x,\nabla g(x)-G(x))$}
\label{sect7}
The representation \eqref{812} of the relaxed energy $J(g,G)$ established in Corollary~\ref{finalfinal} contains the bulk part 
$$
\int_{\Omega } \big(H_p(x,\nabla g(x),G(x))+\Psi (x,\nabla g(x)-G(x))\big)\,\de x,
$$
in which the bulk relaxed density is a sum of the contribution $H_p(x,\nabla g(x),G(x))$ from the initial local energy $E_L(u) $ in  \eqref{901} and the contribution $\Psi (x,\nabla g(x)-G(x))$ from the initial non-local energy $E^{\alpha _{r}}$ in  \eqref{1011}. 
The second term $\Psi (x,\nabla g(x)-G(x))$ has the distinction of capturing a bulk energy density due to disarrangements alone through its sole dependence on the deformation due to disarrangements $M(x)=\nabla g(x)-G(x)$, while the first term $H(x,\nabla g(x),G(x))$ can be written as $H(x,G(x)+M(x),G(x))$ and so depends in general on both the deformation due to disarrangements $M(x)$ and the deformation without disarrangements $G(x)$.
This situation leads naturally to the question of finding conditions on the initial local energy $E_L(u)$ that imply that the term $H_p(x,G(x)+M(x),G(x))$ depends on $G(x)$ alone or, more generally, that 
\begin{equation}\label{decomposition of H}
H_p(x,G(x)+M(x),G(x))=H_{\backslash}(x,G(x))+H_{d}(x,M(x)),
\end{equation}
in which case the bulk relaxed density becomes
\begin{equation}\label{desired decomposition}
H_p(x,G(x)+M(x),G(x))+\Psi (x,M(x))=H_{\backslash}(x,G(x))+(H_{d}(x,M(x))+\Psi (x,M(x))),  
\end{equation}%
a function $H_{\backslash}$ of deformation without disarrangements  plus a function $H_{d}+\Psi$ of deformation due to disarrangements.
The existence of a decomposition of the form \eqref{decomposition of H} was raised in \cite{CF1997} and  \cite{CDPFO1999} and was shown not to be available, in general, in the study \cite{L2000}. 
A modified form of \eqref{decomposition of H} was established in \cite{BMS2012}: there the $x$-dependence was absent, and the term $H_{\backslash}(x,G(x))$ in \eqref{desired decomposition} was replaced by $H_{\backslash}(G(x),\nabla G(x))$. 
In the context of plasticity addressed in the articles \cite{CDPFO1999,DP2018,DO2000}, the availability of \eqref{decomposition of H} was shown to provide a variational basis for describing and predicting the phenomena of yielding, hysteresis, and hardening observed in both single crystals and in polycrystalline materials. 
In this subsection we verify that conditions on the initial local energy $E_{L}(u)$ in \eqref{901} that were identified in \cite[pages~100-101]{CF1997} guarantee the validity of the special additive decomposition \eqref{decomposition of H} for the bulk relaxed/upscaled energy density $H_{1}$ and provide explicit formulas for the functions $H_{\backslash}$ and $H_{d}$ in that decomposition. 
Because no proof of the special additive decomposition \eqref{decomposition of H} is given in \cite{CF1997}, we provide a proof in the context of the following remark that employs the recent results in \cite{S17}. 
In the following remark and in its proof, $A$, $B$, and $A-B$ play the roles of the values of the fields $\nabla g$, $G$, and $M$, respectively.
\begin{remark}\label{remark1}
For the initial local energy $E_{L}(u)$ in \eqref{901}, assume that $W\colon \Omega \times \mathbb{R}^{N\times N}\to \mathbb{R}$ is a continuous function, convex in the second variable, that satisfies $(W1)_1$ and $(W2)$ and that $\psi\colon\Omega \times \mathbb{R}^{N}\times \S{N-1}\to [0,+\infty)$ is continuous, satisfies $(\psi 1)$--$(\psi 4)$, and is such that $\psi(-\lambda,-\nu)=\psi(\lambda,\nu)$.
It follows that the cell formula \eqref{906} for $p=1$ becomes
\begin{equation}\label{decomposition of H1}
H_{1}(x_0,A,B)=W(x_0,B)+\inf \bigg\{ \int_{Q\cap S_{u}}\psi (x_0,[u](x),\nu_{u}(x))\,\de\mathcal{H}^{N-1}(x):u\in \mathcal{C}_{1}^{\bulk}(A,B)\bigg\} 
\end{equation}%
for every $x_0\in \Omega $ and $A,B\in \mathbb{R}^{N\times N}$.
Moreover, the infimum on the right-hand side is given by the expressions
\begin{equation}\label{explicit form}
\begin{split}
H_{1}(x_0,A,B)-W(x_0,B)=\sup \big\{ & \Theta (x_0,A-B) : \Theta (x_0,\cdot )\colon\mathbb{R}^{N\times N}\to[0,+\infty)\;\text{is subadditive and} \\ 
& \Theta (x_0,\lambda\otimes \nu)\leq \psi (x_0,\lambda,\nu)\,\text{for all $\lambda\in \mathbb{R}^{N}$ and $\nu\in\S{N-1}$}
\big\}.
\end{split} 
\end{equation}
\begin{proof}
It is convenient to omit the explicit appearance of the point $x_0\in \Omega$ that appears on both sides of the desired decomposition and that remains fixed throughout the proof.
Let $A,B\in \mathbb{R}^{N\times N}$ and $u\in \mathcal{C}_{1}^{\bulk}(A,B)$ be given.
The cell formula \eqref{906} along with the convexity and continuity of $W$ yield the inequalities
\begin{equation*}
\begin{split}
\int_{Q}  W(\nabla u(x))\,\de x & +\int_{Q\cap S_{u}} \!\!\!\! \psi ([u](x),\nu _{u}(x))\,\de\mathcal{H}^{N-1}(x) \geq  W\bigg(\int_{Q}\nabla u(x)\,\de x\bigg)+\int_{Q\cap S_{u}} \!\!\!\! \psi ([u](x),\nu_{u}(x))\,\de\mathcal{H}^{N-1}(x) \\
&\geq W(B)+\inf \bigg\{ \int_{Q\cap S_{u}}\psi ([u](x),\nu _{u}(x))\,\de\mathcal{H}^{N-1}(x):u\in \mathcal{C}_{1}^{\bulk}(A,B)\bigg\} ,
\end{split}
\end{equation*}%
and, therefore, also yield the lower bound%
\begin{equation}
H_{1}(A,B)\geq W(B)+\inf \bigg\{ \int_{Q\cap S_{u}}\psi ([u](x),\nu_{u}(x))\,\de\mathcal{H}^{N-1}(x):u\in \mathcal{C}_{1}^{\bulk}(A,B)\bigg\} .
\label{lower bound for H1}
\end{equation}%
To obtain an upper bound for $H_{1}(A,B)$ we note from \eqref{T001} that 
$
\mathcal{C}_{1}^{\bulk}(A,B)\supset \big\{ u\in SBV(Q,\mathbb{R}^{N}) :u|_{\partial Q}=Ax, \,\nabla u=B\;\text{$\cL^N$-a.e.~on $Q$}\big\}\eqqcolon\mathcal{C}(A,B)
$,
so that%
\begin{equation}\label{upper_bound_H1}
\begin{split}
&H_{1}(A,B) =\inf \bigg\{ \int_{Q}W(\nabla u(x))\,\de x+\int_{Q\cap S_{u}}\psi ([u](x),\nu _{u}(x))\,\de\mathcal{H}^{N-1}(x):u\in \mathcal{C}_{1}^{\bulk}(A,B)\bigg\} \\
&\leq \inf \bigg\{ \int_{Q}W(\nabla u(x))\,\de x+\int_{Q\cap S_{u}}\psi([u](x),\nu _{u}(x))\,\de\mathcal{H}^{N-1}(x):u\in \mathcal{C}(A,B)\bigg\} \\
&=W(B)+\inf \bigg\{ \int_{Q\cap S_{u}}\psi ([u](x),\nu_{u}(x))\,\de\mathcal{H}^{N-1}(x):u\in \mathcal{C}(A,B)\bigg\}   \\
&=W(B)+\inf \bigg\{ \int_{Q\cap S_{u}} \!\!\!\!\!\! \psi ([u](x),\nu_{u}(x))\,\de\mathcal{H}^{N-1}(x): u\in SBV(Q,\mathbb{R}^{N}), u|_{\partial Q}=(A-B)x, \nabla u=0\,\text{a.e.}\bigg\} \\
&=W(B)+\inf \left\{ \int_{Q\cap S_{u}} \!\!\!\!\!\!\psi ([u](x),\nu_{u}(x))\,\de\mathcal{H}^{N-1}(x): u\in SBV(Q,\mathbb{R}^{N}),u|_{\partial Q}=(A-B)x, \int_{Q}\nabla u=0\right\},  
\end{split}
\end{equation}
where the last equality is established in \cite[Theorem~2.3(iii) and (iv)]{S17}. 
It is now easy to see that the upper bound \eqref{upper_bound_H1} and lower bound \eqref{lower bound for H1} just obtained for $H_{1}(A,B)$ are the same. 
The relation \eqref{explicit form} follows from \eqref{decomposition of H1} and from \cite[Theorem~2.3(i)]{S17}. 
\end{proof}
\end{remark}
\section{Example from crystal plasticity}\label{plasticity}
We turn to the subject of the mechanics of single crystals to identify an example of bulk energies of the type recovered in the volume integral in \eqref{1146} through our combined upscaling and spatial localization of non-local energies.
The example emerges within the special class of \emph{invertible structured deformations} $(g,G)$ in which the tensors $G$ and $K_{(g,G)}$ in \eqref{factorization} below play the role of $F^e$ and $(F^p)^{-1}$ in the standard treatments of crystal plasticity.
\subsection{Invertible structured deformations}
The main mechanisms of deformation in single crystals are the distortion without disarrangements of the crystalline lattice and the shearing due to disarrangements. 
The articles \cite{DPO1993,DO2002a} show that the class of invertible structured deformations is appropriate for capturing such multiscale geometrical changes.
In the present setting, we can identify $(g,G)$ as an \emph{invertible structured deformation} when (see \cite{DPO1993} for a broader setting for this notion)
\begin{itemize}
\item[($I1$)] $g$ is a diffeomorphism of class $C^{1}$ for which $\nabla g$ and $(\nabla g)^{-1}$ are Lipschitzian, 
\item[($I2$)] $G$ is continuous on $\overline\Omega$ with invertible values, 
\item[($I3$)] the macroscopic volume change multiplier $\det \nabla g$ and the multiplier for volume change without disarrangements $\det G$ are equal: $\det \nabla g=\det G$.
\end{itemize} 
For an open set $\Omega\subset\R{3}$, we define 
$$ISD(\Omega;\R{3}\times\R{3\times 3})\coloneqq \big\{(g,G)\in SD(\Omega;\R{3}\times\R{3\times 3}): \text{($I1$), ($I2$), and ($I3$) hold}\big\}.$$
Invertible structured deformations turn out to be a useful setting for understanding some kinematical ingredients in continuum models of single crystals undergoing plastic deformations, partly because the relation $\det \nabla g$ $=\det G$ reflects the fact that the disarrangements occurring in single crystals typically do not involve changes in volume, \emph{i.e.}, arise without the formation of submacroscopic voids. 
One useful mathematical property of invertible structured deformations rests on the notion of composition of invertible structured deformations: if $(g,G)\in ISD(\Omega;\R{3}\times\R{3\times 3})$ and $(h,H)\in ISD(g(\Omega ))$, then the composition $(h,H)\diamond (g,G)$ is defined by%
\begin{equation}\label{definition of composition}
(h,H)\diamond (g,G)\coloneqq(h\circ g,(H\circ g)G).
\end{equation}
It is easy to show that $(h,H)\diamond (g,G)\in ISD(\Omega;\R{3}\times\R{3\times 3})$ and each $(g,G)\in ISD(\Omega;\R{3}\times\R{3\times 3})$ has the factorization
\begin{equation}\label{factorization}
(g,G)=(g,\nabla g)\diamond (i,K_{(g,G)})  
\end{equation}%
where $i\coloneqq x\mapsto x$ is the identity mapping on $\Omega $ and $K_{(g,G)} \coloneqq(\nabla g)^{-1}G$. 
The factor $(g,\nabla g)\in ISD(\Omega;\R{3}\times\R{3\times 3})$ carries all of the macroscopic deformation and is a classical deformation, \emph{i.e.}, it causes no disarrangements because $M_{(g,\nabla g)}\coloneqq \nabla g-\nabla g=0$. 
The factor $(i,K_{(g,G)})\in ISD(\Omega;\R{3}\times\R{3\times 3})$ is purely submacroscopic, \emph{i.e.}, it causes no macroscopic deformation, and carries the
disarrangements
\begin{equation}\label{disarrangement relation}
M_{(i,K_{(g,G)})} \coloneqq \nabla i-K_{(g,G)}=I-(\nabla g)^{-1}G =(\nabla g)^{-1}(\nabla g-G)=(\nabla g)^{-1}M_{(g,G)}.
\end{equation}
Moreover, both factors in (\ref{factorization}) are invertible structured deformations, because $\det K_{(g,G)}=\det G/\det \nabla g=1=\det \nabla i$ and, trivially, $\det \nabla g=\det \nabla g$.
\subsection{Slip systems for single crystals; crystallographic structured deformations}
For a single crystal in the reference configuration $\Omega$ the crystallographic data required for the analysis of crystallographic slip consists of pairs of orthogonal unit vectors $(s^{a},m^{a})$ for $a=1,\ldots,A$, with $A$ the number of potentially active slip systems. 
For crystallographic slip, the discontinuity in deformation arises only across a limited family of slip planes identified via the slip systems. 
The unit vector $s^{a}$ provides the direction of slip, while the unit vector $m^{a}$ is a normal to the slip plane for the $a^{th}$ slip-system $(s^{a},m^{a})$.
For the case of face-centered cubic crystals, the vectors $m^{a}$ are chosen from the normals to the faces of a preassigned regular octahedron and the slip vectors $s^{a}$ are chosen to be one of the directed edges of the face associated with $m^{a}$.
We wish next to identify a collection of invertible structured deformations for which the disarrangements arise only through the action of the slip systems of a give crystal. 
To this end, we recall \cite{DPO1995} that for each structured deformation $(g,G)$ in the sense of \cite{DPO1993} and, hence, for each invertible structured deformation there exists a sequence of injective, piecewise smooth deformations $u_{n}$ such that 
$$u_n\to g\quad\text{in $L^{\infty }(\Omega;\mathbb{R}^{3})$},\qquad \nabla u_{n}\to G\quad\text{in $L^{\infty }(\Omega;\mathbb{R}^{3\times 3})$,}$$
and, for every such sequence and for every $x\in \Omega $, the disarrangement tensor $M_{(g,G)}$ is given by the identification relation 
\begin{equation}\label{identification relation for M}
M_{(g,G)}(x) \coloneqq \nabla g(x)-G(x)= \lim_{r\to 0}\lim_{n\to \infty} \frac1{V_3(r)} \int_{B_r(x)\cap S_{u_{n}}} [u_{n}](y)\otimes \nu_{u_{n}}(y)\,\de \cH^{N-1}(y),  
\end{equation}
and the deformation without disarrangements $G$ is given by the identification relation 
\begin{equation}\label{identification relation for G}
G(x)=\lim_{r\to 0}\lim_{n\to \infty} \frac1{V_3(r)} \int_{B_{r}(x)} \nabla u_{n}(y)\,\de y.
\end{equation}
In both \eqref{identification relation for M} and \eqref{identification relation for G}, $V_3(r)$ denotes the volume of the three-dimensional ball of radius $r$.
Suppose now that the approximating deformations $u_{n}$ are such that the dyadic fields $[u_{n}]\otimes \nu _{u_{n}}$ are compatible with the $A$ slip-systems of the crystal in the sense that for every $n\in \mathbb{N}$ and for every $r>0$ there exist continuous fields $\gamma _{n}^{a}(\cdot,r)\colon\Omega \to \mathbb{R}$ for $a=1,\ldots, A$ such that
\begin{equation}\label{sufficient conditions}
\int_{B_r(x)\cap S_{u_{n}}} [u_{n}](y)\otimes \nu_{u_{n}}(y)\,\de\cH^{N-1}(y)=\sum_{a=1}^{A} V_3(r)\gamma_{n}^{a}(x,r)\nabla g(x)s^{a}\otimes m^{a},  
\end{equation}
and such that $\lim_{r\to 0}\lim_{n\to \infty }\gamma_{n}^{a}(x,r)\eqqcolon\gamma ^{a}(x)$ exists for every $x\in \Omega $ and for $a=1,\ldots, A$.
Under these assumptions, the identification relation \eqref{identification relation for M} becomes 
$$ M_{(g,G)} (x) =\sum_{a=1}^{A}\gamma ^{a}(x)\nabla g(x)s^{a}\otimes m^{a}, $$
so that the first relation in \eqref{disarrangement relation} becomes 
\begin{equation}\label{crystallographic disarrangements}
M_{(i,K_{(g,G)})}\left( x\right) =I-K_{(g,G)}(x)=\sum_{a=1}^{A}\gamma^{a}(x)s^{a}\otimes m^{a}  
\end{equation}%
and, therefore, 
\begin{equation}\label{K for crystallographic deformations}
K_{(g,G)}(x)=I-\sum_{a=1}^{A}\gamma ^{a}(x)s^{a}\otimes m^{a}.
\end{equation}%
We have provided through \eqref{sufficient conditions} sufficient conditions that the disarrangement tensor field $M_{(i,K_{(g,G)})}$ for the purely submacroscopic part $(i,(\nabla g)^{-1}G)$ of $(g,G)$ is a linear combination of the (spatially constant) crystallographic slip dyads $s^{a}\otimes m^{a}$ for $a=1,\ldots, A$ associated with the given crystal.
In this context we may say that the invertible structured deformation $(g,G)$ generates disarrangements only in the form of crystallographic slips or, more briefly, that $(g,G)$ is \textit{crystallographic}. 
We note in passing that the article \cite{DO2002a} provided a precise sense in which one may consider \emph{approximations} by crystallographic slips of the disarrangement matrix $M_{(g,G)}=\nabla g-G$ of any invertible structured deformations.
Since we here restrict our attention to those invertible structured deformations for which \eqref{crystallographic disarrangements} holds, the approximations in \cite{DO2002a} become exact in the present context.
For a crystallographic structured deformation $(g,G)$ and a point $x\in\Omega $ we say that a slip-system $a$ is active at $x$ if $\gamma^{a}(x)\neq 0$, and we say that single slip occurs at $x$ if there is only one slip-system that is active at $x$. 
If more than one slip system is active at $x$ we say that multiple slip occurs at $x$. 
If $(g,G)$ is crystallographic, so that \eqref{sufficient conditions}, \eqref{crystallographic disarrangements}, and \eqref{K for crystallographic deformations} hold, we may use the relations $\tr(s^{a}\otimes m^{a})=s^{a}\cdot m^{a}=0$ for all $a$ and $\det K_{(g,G)}=\det G/\det \nabla g$ to conclude from \eqref{K for crystallographic deformations} and the definition of invertible structured deformations that
\begin{equation}\label{restrictions on K}
\tr K_{(g,G)}=3\qquad\text{and}\qquad \det K_{(g,G)}=1.
\end{equation}%
Consequently, the crystallographic structured deformations are among those for which $K_{(g,G)}=(\nabla g)^{-1}G$ satisfies \eqref{restrictions on K}.
We note that a slip system $a$ is active at $x$ for $(g,G)$ if and only if $a$ is active at $x$ for the purely submacroscopic part $(i,K_{(g,G)})$ of $(g,G)$.

Examples of crystallographic structured deformations that undergo single slip at every point are the two-level shears $(g_{\mu ,x_{o}}^{a},G_{\gamma}^{a})$ for $a=1,\ldots, A,$ for $\mu ,\gamma \in \mathbb{R}$ and for $x_{o}\in\Omega $:
\begin{equation}\label{two-level shears}
\begin{split}
g_{\mu ,x_{o}}^{a}(x) \coloneqq &\; x_{o}+(I+\mu s^{a}\otimes m^{a})(x-x_{o})  \\
G_{\gamma}^{a}(x) \coloneqq &\; I+\gamma s^{a}\otimes m^{a},  
\end{split}
\end{equation}
for which it can be verified \cite{DPO1993} via the ``deck of cards'' family of approximations $u_{n}$ that \eqref{sufficient conditions} is satisfied, and for which 
\begin{subequations}
\begin{eqnarray}
\nabla g_{\mu}^{a}(x) &\!\!\!\!= &\!\!\!\! I+\mu s^{a}\otimes m^{a},  \label{g a mu} \\
M_{(g_{\mu ,x_{o}}^{a},G_{\gamma}^{a})}(x) &\!\!\!\!=&\!\!\!\! (\mu -\gamma)s^{a}\otimes m^{a}, \label{M a mu nu} 
\end{eqnarray}%
\end{subequations}
and
\begin{equation}\label{K a mu nu}
K_{(g_{\mu ,x_{o}}^{a},G_{\gamma}^{a})}(x)=I+(\gamma -\mu)s^{a}\otimes m^{a}=I-M_{(g_{\mu ,x_{o}}^{a},G_{\gamma}^{a})}(x).  
\end{equation}%
The ``deck of cards'' approximations $f_{n}$ show that each two-level shear $(g_{\mu ,x_{o}}^{a},G_{\nu }^{a})$ is approximated for each $n$ by smooth shears of amount $\gamma$ of the crystal lattice between $n-1$ slip planes, along with slip-discontinuities in the direction $s^{a}$ across the $n-1$ planes, each slip-discontinuity of amount $\frac{\mu -\gamma}{n}$ times a reference dimension in the direction $m^{a}$. 
By virtue of the ``deck of cards'' approximations $u_{n}$ and, in view of \eqref{g a mu}, \eqref{two-level shears}$_{2}$, \eqref{M a mu nu}, and the trivial relation
\begin{equation*}
\mu =\gamma+(\mu -\gamma),
\end{equation*}
we may then call $\mu $ the macroscopic shear, $\gamma$ the shear without slip, and $\mu -\gamma$ the shear due to slip for the two-level shear $(g_{\mu,x_{o}}^{a},G_{\gamma}^{a})$. 
Of particular interest is the case $\gamma=0$, \emph{i.e.}, the two-level shear $(g_{\mu ,x_{o}}^{a},I)$, in which the region between slip planes undergoes no shear and the macroshear $\mu $ arises entirely from slips on slip-system $a$. 
\subsection{Slip-neutral two-level shears}
We now summarize arguments provided in \cite{CDPFO1999} in a more limited setting that are based on the observation that crystallographic slip is physically activated within very thin bands, the so-called slip-bands, whose thickness is typically of the order $10^{2}$ atomic units, while the separation of active slip-bands is typically of order $10^{4}$ atomic units. 
The arguments in \cite{CDPFO1999} indicate the following: for each $a=1,\ldots, A$, there is a number $p^{a}>0$ such that a two-level shear $(g_{\mu,x_{o}}^{a},G_{\gamma}^{a})$ for which the shear due to slip $\mu -\gamma $ is an integral multiple of $p^{a}$ gives rise to submacroscopic slips equal to an integral number of atomic units in the direction of slip $s^{a}$.
The dimensionless number $p^{a}$ equals a shift of one atomic unit in the direction of slip $s^{a}$ divided by the separation $10^4$ in the direction $m^{a}$ of consecutive active slip-bands associated with system $a$ (measured in the same atomic units). 
Consequently, $p^{a}$ is of the order of $10^{-4}$, and a two-level shear $(g_{\mu ,x_{o}}^{a},G_{\gamma}^{a})$ with
\begin{equation}\label{neutral slips}
\mu -\gamma =np^{a},\qquad\text{with $n\in \mathbb{Z}$,}
\end{equation}%
produces a shift of $n$ atomic units and so does not produce a misfit of the crystalline lattice across the active slip bands, no matter what the amount of shear without slip $\gamma$.  
Thus, when \eqref{neutral slips} holds, the disarrangements due to slip are not revealed by the deformed positions under the two-level shear attained by the lattice points away from the slip bands.
We refer to a two-level shear $(g_{\mu ,x_{o}}^{a},G_{\gamma}^{a})$ satisfying \eqref{neutral slips} as slip-neutral for the slip-system $a$. 
In particular, when $\gamma=0$ we have $G_{\nu }^{a}=I$, and the two-level shear $(g_{\mu }^{a},I)$ is slip-neutral if the macroshear $\mu =\mu -\gamma$ is an integral multiple of $p^{a}$. 
Although a slip-neutral shear of the form $(g_{np^{a}}^{a},I)$ causes a macroscopic shearing of the body, not only does it cause no misfit of the lattice across slip bands, it also causes no distortion of the lattice. 
Consequently, we call the two-level shear $(g_{np^{a}}^{a},I)$ completely neutral for the slip-system $a$.
We suppose now that the given body undergoes a completely neutral two-level shear $(g_{\mu ,x_{o}}^{a},I)$ with $\mu =np^{a}$, starting from the region $\Omega $, and suppose further that $(g_{\mu ,x_{o}}^{a},I)$ is then followed by a crystallographic deformation $(g,G)$, so that we have the composition and factorization as in \eqref{definition of composition} and \eqref{factorization}: 
$$
(g,G)\diamond (g_{\mu ,x_{o}}^{a},I) =(g\circ g_{\mu ,x_{o}}^{a},G\circ
g_{\mu ,x_{o}}^{a}) 
=(g\circ g_{\mu ,x_{o}}^{a},\nabla (g\circ g_{\mu ,x_{o}}^{a}))\diamond
\big(i,K_{(g\circ g_{\mu ,x_{o}}^{a},G\circ g_{\mu ,x_{o}}^{a})}\big),
$$
with $K_{(g\circ g_{\mu ,x_{o}}^{a},G\circ g_{\mu ,x_{o}}^{a})}$ given by%
\begin{equation*}
\begin{split}
K_{(g\circ g_{\mu ,x_{o}}^{a},G\circ g_{\mu ,x_{o}}^{a})} =&\; (\nabla (g\circ g_{\mu ,x_{o}}^{a}))^{-1}(G\circ g_{\mu ,x_{o}}^{a}) 
=\big((\nabla g\circ g_{\mu ,x_{o}}^{a})\nabla g_{\mu ,x_{o}}^{a}\big)^{-1}(G\circ g_{\mu ,x_{o}}^{a}) \\
=&\;(\nabla g_{\mu ,x_{o}}^{a})^{-1}(\nabla g\circ g_{\mu,x_{o}}^{a})^{-1}(G\circ g_{\mu ,x_{o}}^{a}) 
=\big((\nabla g_{\mu ,x_{o}}^{a})^{-1}(\nabla g)^{-1}G\big)\circ g_{\mu ,x_{o}}^{a}
\\
=&\; \big(\nabla g_{\mu ,x_{o}}^{a})^{-1}K_{(g,G)}\big)\circ g_{\mu ,x_{o}}^{a}
\end{split}
\end{equation*}%
and with $K_{(g_{\mu ,x_{o}}^{a},I)}$ given by (\ref{K a mu nu}):%
$$
K_{(g_{\mu ,x_{o}}^{a},I)}=I-\mu s^{a}\otimes m^{a}=(\nabla g_{\mu
,x_{o}}^{a})^{-1}.
$$
Therefore, we have the relation%
\begin{equation*}
K_{(g,G)\diamond (g_{mp^{a}}^{a},I)}=K_{(g_{\mu,x_{o}}^{a},I)}(K_{(g,G)}\circ g_{\mu ,x_{o}}^{a}),
\end{equation*}%
and the relations $M_{(i,K_{(g,G)})}=I-K_{(g,G)}$, \eqref{K for crystallographic deformations}, \eqref{M a mu nu}, and \eqref{K a mu nu} then yield 
\begin{equation}\label{long computation}
\begin{split}
I-M_{(i,K_{(g,G)\diamond (g_{mp^{a}}^{a},I)})} =&\; K_{(g,G)\diamond (g_{mp^{a}}^{a},I)}  = K_{(g_{\mu ,x_{o}}^{a},I)}(K_{(g,G)}\circ g_{\mu ,x_{o}}^{a}) \\
=&\; (I-M_{_{(g_{\mu ,x_{o}}^{a},I)}})(I-M_{(i,K_{(g,G)})}\circ g_{\mu,x_{o}}^{a})  \\
=&\; I-M_{(i,{K}_{(g,G)})}\circ g_{\mu ,x_{o}}^{a}-M_{(g_{\mu,x_{o}}^{a},I)}+M_{_{(g_{\mu ,x_{o}}^{a},I)}}M_{(i,K_{(g,G)})}\circ g_{\mu
,x_{o}}^{a}   \\
=&\; I-M_{(i,{K}_{(g,G)})}\circ g_{\mu ,x_{o}}^{a}-M_{(g_{\mu,x_{o}}^{a},I)} +\mu (s^{a}\otimes m^{a})\,M_{(i,K_{(g,G)})}\circ g_{\mu ,x_{o}}^{a}  \\
=&\; I-M_{(i,{K}_{(g,G)})}\circ g_{\mu ,x_{o}}^{a}-M_{(g_{\mu,x_{o}}^{a},I)} +\mu s^{a}\otimes (M_{(i,K_{(g,G)})}^{T}\circ g_{\mu ,x_{o}}^{a})\,m^{a}.
\end{split}
\end{equation}%
When $\mu \neq 0$ the last term in \eqref{long computation} vanishes at a point $x$ if and only if $M_{(i,K_{(g,G)})}^{T}(g_{\mu,x_{o}}^{a}(x))\,m^{a}=0$, and, because $g_{\mu ,x_{o}}^{a}(x_{o})=x_{o}$, we conclude from \eqref{crystallographic disarrangements} the following remark.
\begin{remark}\label{previous}
The disarrangement tensor $M_{(i,K_{(g,G)\diamond (g_{\mu,x_{o}}^{a},I)})}(x_{o})$ at $x_{o}\in \Omega $ for the submacroscopic part of the composition $(g,G)\diamond (g_{\mu ,x_{o}}^{a},I)$ of a crystallographic deformation $(g,G)$ with the completely neutral two-level shear $(g_{\mu ,x_{o}}^{a},I)$, where $\mu =np^{a}$, is given by
\begin{equation}\label{M is a morphism}
M_{(i,K_{(g,G)\diamond (g_{\mu,x_{o}}^{a},I)})}(x_{o})=M_{(i,K_{(g,G)})}(x_{o})+M_{(g_{\mu,x_{o}}^{a},I)}(x_{o})  
\end{equation}
if and only if
\begin{equation}\label{slip condition}
\sum_{b=1}^{A}\gamma ^{b}(x)(s^{b}\cdot m^{a})m^{b}=M_{(i,K_{(g,G)})}^{T}(x_{o})\,m^{a}=0.  
\end{equation}%
The identification relation \eqref{identification relation for M} for $M$ shows that the vanishing of $M_{(i,K_{(g,G)})}^{T}(x_{o})\,m^{a}$ in \eqref{slip condition} is the statement that, on average, as $n\to \infty$ and $r\to 0$, the jumps in approximating deformations $u_{n}$ must be parallel to the slip plane for the $a^\mathrm{th}$ slip system. 
A sufficient condition on the crystallographic deformation $(g,G)$ in order that the sum in \eqref{slip condition} vanish is the following: every slip system $b$ that is active at $x_{o}$ for $(g,G)$ satisfies $s^{b}\cdot m^{a}=0$, \emph{i.e.}, the slip plane for the completely neutral two-level shear $(g_{np^{a}}^{a},I)$ contains every slip direction $s^{b}$ of every slip system $b$, active at $x_{o}$ for $(g,G)$. 
In particular, if $(g,G)$ is a double slip at $x_{o}$ with active slip systems $(s^{1},m^{1})$ and $(s^{2},m^{2})$, then \eqref{M is a morphism} holds for every $a$ such that $s^{1}\cdot m^{a}=s^{2}\cdot m^{a}=0$. 
Such double slips $(g,G)$ include the case of ``cross slip'' in which $(s^{1},m^{1})=(s^{a},m^{1})$ and $(s^{2},m^{2})=(s^{a},m^{2})$ with $m^{a}=m^{1}\neq m^{2}$ in which slips in one and the same direction $s^{a}$ occur in two different slip systems at $x_{o}$. 
\qed
\end{remark}
Our discussion above of the relationship between the disarrangement tensor $M_{(i,K_{(g,G)\diamond (g_{\mu ,x_{o}}^{a},I)})}$ for the purely submacroscopic part of the composition $(g,G)\diamond (g_{\mu ,x_{o}}^{a},I)$ and the disarrangement tensor $M_{(i,K_{(g,G)})}$ for the purely submacroscopic part of $(g,G)$ is of particular interest for energetics, because $(g_{\mu ,x_{o}}^{a},I)$ was assumed to be completely neutral for the slip-sytem $a$, \emph{i.e.}, $\mu =np^{a}$ with $n$ an integer.
In that case, the lattice on which $(g,G)$ acts when following $(g_{\mu,x_{o}}^{a},I)$ differs from that on which $(g,G)$ acts when \textit{not}
following $(g_{\mu ,x_{o}}^{a},I)$ only by undetectable translations of the lattice between active slip-planes for system $a$. 
Consequently, the submacroscopic kinematical states of the crystal attained by means of the two purely submacroscopic structured deformations $(i,K_{(g,G)\diamond (g_{\mu ,x_{o}}^{a},I)})$ and $(i,K_{(g,G)})$ are indistinguishable. 
Therefore, the energetic responses to the corresponding disarrangement fields $M_{(i,K_{(g,G)\diamond(g_{\mu ,x_{o}}^{a},I)})}$ and $M_{(i,K_{(g,G)})}$ would be indistinguishable, so that the validity of \eqref{M is a morphism} would have significant implications with respect to properties of the energetic response of the crystal.
We now provide specific circumstances under which the relaxed energies recovered in Corollary~\ref{finalfinal} would be subject to those implications, and we set the stage by highlighting the role of $M_{(i,K_{(g,G)})}$, the disarrangement tensor for the purely submacroscopic deformation $(i,K_{(g,G)})$, in providing constitutive relations for the energetic response to crystallographic deformation that are frame-indifferent (independent of observer).
\subsection{Frame-indifferent energetic responses}
We noted in the text above the relation \eqref{disarrangement relation} that contains the formulas
$$
M_{(i,K_{(g,G)})}=I-K_{(g,G)}=(\nabla g)^{-1}M_{(g,G)},
$$
relating $M_{(g,G)}$, the disarrangement tensor for an invertible structured deformation $(g,G)$, and $M_{(i,K_{(g,G)})}$, the disarrangement tensor for the purely submacroscopic deformation $(i,K_{(g,G)})$ in \eqref{factorization}. 
Because $\nabla g$ and $M$ both are premultiplied by a rotation $Q$ under a change of observer associated with the rotation $Q$, the tensor field $K_{(g,G)}$ as well as the disarrangement tensor $M_{(i,K_{(g,G)})}$ are unchanged under such a change of observer. 
Therefore, for a function $\Psi_{i}:\mathbb{R}^{3\times 3}\rightarrow \mathbb{R}$ the mapping 
\begin{equation*}
x \mapsto \Psi _{i}(M_{(i,K_{(g,G)})}(x)) =\Psi _{i}((\nabla g(x))^{-1}M_{(g,G)}(x))  
\end{equation*}%
has the property that its dependence on the structured deformation $(g,G)$ is independent of observer. 
The function $\Psi _{i}$ specifies the energetic response of a body from the reference cofiguration $\Omega $ to the disarrangements arising in purely submacroscopic deformations. 
If we define for the given macroscopic deformation $g$ the mapping $\Psi_{g}:\Omega \times \mathbb{R}^{3\times 3}\to \mathbb{R}$ by,%
\begin{equation}\label{psi of x and M}
\Psi _{g}(x,L):=\Psi _{i}((\nabla g(x))^{-1}L),\qquad\text{for all $L\in \mathbb{R}^{3\times 3}$,}  
\end{equation}%
then the mapping $x\mapsto \Psi _{g}(x,M_{(g,G)}(x))=\Psi _{i}((\nabla g(x))^{-1}M_{(g,G)}(x))$ also has the property that its dependence on $(g,G)$ is independent of observer. 
The following constitutive assumption for the dependence on invertible structured deformations $(g,G)$ of $\psi \colon\Omega\to \mathbb{R}$, the free energy density due to disarrangements, namely,
\begin{equation}\label{final constitutive assumption}
\psi (x) =\Psi _{i}(M_{(i,K_{(g,G)})}(x))=\Psi _{i}((\nabla g(x))^{-1}M_{(g,G)}(x)) =\Psi _{g}(x,M_{(g,G)}(x)),\qquad x\in \Omega,
\end{equation}%
then is independent of observer and carries the assumption that the free energy density due to disarrangements depends only on the disarrangements associated with the submacroscopic factor $(i,K_{(g,G)})$ in \eqref{factorization}. 
When $(g,G)$ is a crystallographic deformation, then the response functions $\Psi _{i}$ and $\Psi _{g}$ determine the free energy density due to crystallographic slip as a function of the disarrangement tensors $M_{(i,K_{(g,G)})}$ and $M_{(g,G)}$, respectively.
\subsection{Periodic properties of the energetic response $\Psi _{i}$ to crystallographic slip}
Let $x_{o}\in \Omega $, $a\in \{1,\ldots, A\}$, $\mu =np^{a}$ with $n\in \mathbb{Z}$, and a crystallographic structured deformation $(g,G)$ be given. 
We argued above that the lattice on which $(g,G)$ acts, when following the completely neutral two-level shear $(g_{\mu ,x_{o}}^{a},I)$, differs from that on which $(g,G)$ acts, when \textit{not} following $(g_{\mu,x_{o}}^{a},I)$, only by the undetectable translations of the lattice between active slip bands for system $a$. 
Consequently, the submacroscopic kinematical states of the crystal lattice attained by means of the two purely submacroscopic structured deformations $(i,K_{(g,G)\diamond (g_{\mu,x_{o}}^{a},I)})$ and $(i,K_{(g,G)})$ are indistinguishable. 
We invoke this indistinguishability to assert that the free energy density $\psi(x_{o})$ due to crystallographic slip should be the same for $(i,K_{(g,G)\diamond (g_{\mu ,x_{o}}^{a},I)})$ and for $(i,K_{(g,G)})$ at the fixed point $x_{o}$ of $g_{\mu ,x_{o}}^{a}$. 
Under the constitutive assumption \eqref{final constitutive assumption} applied to the point $x_{o}$ this assertion means that, for every $\mu =np^{a}$ with $n\in \mathbb{Z}$,
\begin{equation}\label{restriction on PHI}
\Psi _{i}(M_{(i,K_{(g,G)\diamond (g_{\mu ,x_{o}}^{a},I)})}(x_{o}))=\Psi_{i}(M_{(i,K_{(g,G)})}(x_{o})). 
\end{equation}
We wish to translate \eqref{restriction on PHI} into a property of the response function $\Psi _{i}\colon \mathbb{R}^{3\times 3}\to\mathbb{R}$ by invoking the additivity property \eqref{M is a morphism} in Remark~\ref{previous}. 
This property requires that we restrict attention to matrices $M\in \mathbb{R}^{3\times 3}$ of the form
\begin{equation}\label{M crystallographic}
M=\sum\limits_{b=1}^{A}\beta ^{b}s^{b}\otimes m^{b}
\end{equation}%
with $\beta ^{1},\ldots, \beta^{A}\in \mathbb{R}$ such that 
\begin{equation}\label{compatibility with a}
M^\top m^a=0,
\end{equation}%
and such that%
\begin{equation}\label{volume constraint}
\det (I-M)=1.
\end{equation}%
If we define%
\begin{equation*}
\M{a} \coloneqq \{M\in \mathbb{R}^{3\times 3} : \text{\eqref{M crystallographic}, \eqref{compatibility with a}, \eqref{volume constraint} hold}\}, 
\end{equation*}%
then it is easy to show that if there exists $b\in \{ 1,\ldots, A\} $, $s\in \mathbb{R}^{3}$, $\xi \in \mathbb{R}$, such that $s\cdot m^{a}=0$ and 
\begin{equation}
M=s\otimes m^{a}+\xi (m^{a}\times m^{b})\otimes m^{b}
\label{generic M in M^a}
\end{equation}%
then $M\in \M{a}$. 
When $m^{b}=\pm m^{a}$ then the matrix $M$ in \eqref{generic M in M^a} reduces to $s\otimes m^{a}$ and represents disarrangements arising from slips in the crystallographic plane with normal  $m^{a}$, but not necessarily in one of the slip directions in the list of slip systems for the crystal. 
When $m^{b}\neq \pm m^{a}$, $s^a=m^a\times m^b=s$, and $\xi \neq 0$, $M$ represents disarrangements of the previous type along with slips in the direction $m^{a}\times m^{b}$ in the crystallographic plane with normal $m^{b}$ and so corresponds to the cross-slip described in Remark~\ref{remark1}.
Suppose now that $(g,G)$, $a$, and $x_{o}$ are such that $M_{(i,K_{(g,G)})}(x_{o})\in \M{a}$. 
By Remark~\ref{previous}, \eqref{M is a morphism} holds for every completely neutral two-level shear $(g_{np^{a},x_{o}}^{a},I)$, \emph{i.e.}, 
$$
M_{(i,K_{(g,G)\diamond
(g_{np^{a},x_{o}}^{a},I)})}(x_{o})=M_{(i,K_{(g,G)})}(x_{o})+M_{(g_{np^{a},x_{o}}^{a},I)}(x_{o})
$$
which by \eqref{M a mu nu} we may write in the following form
$$
M_{(i,K_{(g,G)\diamond (g_{\mu,x_{o}}^{a},I)})}(x_{o})=M_{(i,K_{(g,G)})}(x_{o})+np^{a}s^{a}\otimes m^{a}.
$$
Consequently, when $M_{(i,K_{(g,G)})}(x_{o})\in \M{a}$, this formula and the constitutive restriction \eqref{restriction on PHI} on the response function $\Psi _{i}$ yield the relation%
\begin{equation}\label{invariance of PHI}
\Psi _{i}\big( M_{(i,K_{(g,G)})}(x_{o})+np^{a}s^{a}\otimes m^{a}\big) =\Psi _{i}\big( M_{(i,K_{(g,G)})}(x_{o})\big),\qquad\text{for every $n\in  \mathbb{Z}$.}  
\end{equation}%
For matrices $M\in $ $\M{a}$ satisfying \eqref{generic M in M^a} this restriction takes the form 
\begin{equation}\label{periodicity condition}
\Psi _{i}((s+np^{a}s^{a})\otimes m^{a}+\xi (m^{a}\times m^{b})\otimes m^{b}) = \Psi _{i}(s\otimes m^{a}+\xi (m^{a}\times m^{b})\otimes m^{b})
\end{equation}%
for every $s\in \{m^{a}\}^{\bot }$, $b\in \{ 1,\ldots, A\} $, $\xi \in \mathbb{R}$, and $n\in \mathbb{Z}$. 
In other terms, \eqref{periodicity condition} is the assertion that for each $b\in \{1,\ldots, A\} $ and $\xi \in \mathbb{R}$ the mapping%
\begin{equation}\label{period pa}
s\mapsto \Psi _{i}(s\otimes m^{a}+\xi (m^{a}\times m^{b})\otimes m^{b})
\end{equation}%
is periodic on $\{m^{a}\}^{\bot }$ with (vector) period $p^{a}s^{a}$.
Thus, the presence of completely neutral two-level shears $(g_{p^{a}}^{a},I)$ has led via \eqref{invariance of PHI} to the identification of a family of affine subspaces%
\begin{equation*}
\M{a}_{b,\xi } \coloneqq \{s\otimes m^{a}+\xi (m^{a}\times m^{b})\otimes m^{b} :\ s\in \{m^{a}\}^{\bot }\}  
\end{equation*}%
of $\mathbb{R}^{3\times 3}$, each two-dimensional and on each of which the restriction of $\Psi _{i}$ is periodic with corresponding period $p^{a}s^{a}$. 
\subsection{Form of the initial non-local energy appropriate for crystalline plasticity}
In this subsection we take the basic constitutive assumption \eqref{final constitutive assumption} and the property \eqref{invariance of PHI} of $\Psi _{i}$ that reflects the complete neutrality of certain two-level shears, and we identify additional properties of $\Psi _{i}$ that permit the application of Theorem~\ref{final1E} when $\Psi _{g}$ appears in place of $\Psi $ in the formula \eqref{1011} for the averaged interfacial energy. 
The following theorem provides conditions on $\Psi _{i}$ and $(g,G)$ sufficient for the application of Theorem~\ref{final1E} in the context of crystal plasticity. 
We note in advance that the fact that the macroscopic deformation $g$ for a crystallographic structured deformation $(g,G)$ is smooth (as is the case, more generally, for $(g,G)\in ISD(\Omega;\R{3}\times\R{3\times 3})$) means that the singular part $D^{s}g$ of the distributional derivative $Dg $ is zero and, consequently, that the term in \eqref{1146} involving the recession function $\Psi _{g}^{\infty }$ is zero. 
(This would not be the case were one to use the original definition of invertible structured deformation in \cite{DPO1993} in which $g$ is allowed to be discontinuous.)
\begin{theorem}
Let $\Omega \subset \mathbb{R}^{3}$ be a bounded domain with Lipschitz boundary, and let $\Psi _{i}\colon \mathbb{R}^{3\times 3}\to \mathbb{R}$ be a sublinear Lipschitz continuous mapping satisfying, for each $a,b\in \{1,\ldots, A\}$ and for each $\xi \in \mathbb{R}$, the periodicity condition \eqref{period pa}. 
Moreover, for each crystallographic structured deformation $(g,G)$, let $\Psi _{g}\colon\Omega\times \mathbb{R}^{3\times 3}\to\mathbb{R}$ be given in terms of $\Psi _{i}$ by \eqref{psi of x and M}, and for each $u\in SBV(\Omega;\mathbb{R}^{3})$ define as in \eqref{1009} the averaged interfacial energy
\begin{equation*}
E_{g}^{\alpha _{r}}(u)\coloneqq \int_{\Omega _{r}}\Psi _{g}(x,(D^{s}u\ast\alpha _{r})(x))\,\de x.  
\end{equation*}%
Then 
if $\alpha_r$ is as in \eqref{1014}, 
the upscaled energy $I^{\alpha _{r}}(g,G;\Omega_r)$ in \eqref{864} is given by 
\begin{equation}\label{relaxed energy for fixed r}
I^{\alpha _{r}}(g,G;\Omega_r)=\int_{\Omega _{r}}\Psi _{g}\big(x,((\nabla g-G) * \alpha _{r})(x)\big)\,\de x,  
\end{equation}%
and the spatially localized, upscaled energy $I(g,G)$ in \eqref{1013} takes the form given in \eqref{624}
\begin{equation}\label{final relaxed energy}
\begin{split}
{I}(g,G) =\lim_{r\to 0^+} I^{\alpha _{r}}(g,G;\Omega_r) =&\; \int_{\Omega }\Psi _{g}(x,\nabla g(x)-G(x))\,\de x=\int_{\Omega}\Psi _{i}(I-\nabla g(x)^{-1}G(x))\,\de x \\
=&\; \int_{\Omega }\Psi _{i} \big(I-K_{(g,G)}(x)\big)\,\de x=\int_{\Omega}\Psi _{i} \big(M_{(i,K_{(g,G)})}(x)\big)\,\de x.  
\end{split}
\end{equation}
In particular, the bulk density for $I(g,G)$ is frame-indifferent and retains the periodicity property \eqref{period pa}.
\end{theorem}
\begin{proof}
We note that for each $(g,G)\in ISD(\Omega;\R{3}\times\R{3\times 3})$ there also holds $(g,G)\in SD(\Omega;\R{3}\times\R{3\times 3})$. 
Moreover, the Lipschitz continuity of $(\nabla g)^{-1}$ and the assumed Lipschitz continuity of $\Psi _{i}$ imply that $\Psi _{g}$ 
belongs to the class (L), so that we may invoke 
Theorem~\ref{mainrho} to obtain \eqref{relaxed energy for fixed r} 
and Theorem~\ref{final1E} to obtain \eqref{final relaxed energy}. 
\end{proof}

We close by noting that, for the case $p=1$, the formulas \eqref{812}, \eqref{decomposition of H1}, \eqref{explicit form}, and \eqref{final relaxed energy} provide the desired, fully three-dimensional analogue of the decomposition \eqref{decomposition for energy}.

\bigskip
\noindent
\textbf{Acknowledgments}: 
Tha authors thank the Center for Nonlinear Analysis at Carnegie Mellon University, Instituto Superior T\'ecnico at Universidade de Lisboa, the Zentrum Mathematik of the Technische Universit\"at M\"unchen, the Dipartimento di Ingegneria Industriale of the Universit\`a di Salerno, and the Dipartimento di Scienze Matematiche ``G.~L.~Lagrange'' of Politecnico di Torino for their support and hospitality. 
Finally, the authors would like to thank the Isaac Newton Institute for Mathematical Sciences for support and hospitality during the programme \emph{The mathematical design of new materials} when work on this paper was undertaken. This work was partially supported by EPSRC Grant Number EP/R014604/1.
MM and EZ are members of the Gruppo Nazionale per l'Analisi Matematica, la Probabilit\`a e le loro Applicazioni (GNAMPA) of the Istituto Nazionale di Alta Matematica (INdAM).
JM acknowledges partial support from the Funda\c{c}\~{a}o para a Ci\^{e}ncia e a Tecnologia through the grant UID/MAT/04459/2013 and gratefully acknowledges support from GNAMPA-INdAM through \emph{Programma professori visitatori 2018}.
MM acknowledges partial support from the ERC Starting grant \emph{High-Dimensional Sparse Optimal Control} (Grant agreement no.~306274) and the DFG Project \emph{Identifikation von Energien durch Beobachtung der zeitlichen Entwicklung von Systemen} (FO 767/7).
EZ acknowledges partial support from the INdAM-GNAMPA Project 2019 \emph{Analysis and optimisation of thin structures}.
Funding from the \emph{Starting grant per giovani ricercatori} of Politecnico di Torino is also gratefully acknowledged.

\bibliographystyle{plain}

\end{document}